\documentclass[final,1p,times]{elsarticle}
\usepackage[colorlinks=true]{hyperref}
\usepackage{amsfonts}
\usepackage[all]{xy}
\usepackage[mathscr]{eucal}
\usepackage{extarrows}
\usepackage{amsthm,amsmath,amssymb,mathrsfs,amsfonts,graphicx,graphics,latexsym,exscale,cmmib57,dsfont,bbm,amscd}
\usepackage{color,soul}%%%\hl{ highlighting text } % use hl{ } to highlight the sentence.
\newtheorem{theorem}{Theorem}[section]
\newtheorem{lemma}{Lemma}[section]

\newtheorem{proposition}{Proposition}[section]
\newtheorem{remark}{Remark}[section]
\theoremstyle{definition}
\newtheorem{definition}{Definition}[section]

\numberwithin{equation}{section}
\journal{xxx}

\begin{document}

\begin{frontmatter}
\title{Local intricacy and average sample complexity for amenable group actions}

\author{Jinna Huang}
\ead{huangjinna63@gmail.com}
\address{Department of mathematics, Soochow University, Suzhou 215006, People's Republic of China}
\author{Zubiao Xiao\corref{cor1}}
\ead{xzb2020@fzu.edu.cn}
\cortext[cor1]{Corresponding author}
\address{School of Mathematics and Statistics, Fuzhou University, Fuzhou 350116, People's Republic of China}

%%%%%%%%%%%%%%%%%%%%%%%%%%%%%%
\begin{abstract}
Let $(X,G)$, $(Y,G)$ be two $G$-systems, where $G$ is an infinite countable discrete amenable group and $X$, $Y$ are compact metric spaces. Suppose that $\mathcal{U}$ is a cover of $X$. We first introduce the conditional local topological intricacy $\mathrm{Int}_\mathrm{top} (G,\mathcal{U}|Y)$ and average sample complexity $\mathrm{Asc}_\mathrm{top} (G,\mathcal{U}|Y)$. Given an invariant measure $\mu$ of $X$, we study the conditional local measure-theoretical intricacy $\mathrm{Int}_\mu^\pm(G,\mathcal{U}|Y)$ and average sample complexity $\mathrm{Asc}_\mu^\pm(G,\mathcal{U}|Y)$. For any F{\o}lner sequence $\{F_n\}_{n\in\mathbb{N}}$, we take  $\{c^{F_n}_S\}_{S\subseteq F_n}$ to be the uniform system of coefficients. We establish the equivalence of $\mathrm{Asc}_\mu^-(G,\mathcal{U}|Y)$ and $\mathrm{Asc}_\mu^+(G,\mathcal{U}|Y)$ when $G=\mathbb{Z}$. Furthermore, we verified that  $\mathrm{Asc}_\mu^-(G,\mathcal{U})$ is equal to $\mathrm{Asc}_\mu^+(G,\mathcal{U})$ in general case. Finally, we give a local variational principle of average sample complexity.
\end{abstract}

\end{frontmatter}

%%% ----------------------------------------------------------------------
%%% ----------------------------------------------------------------------
%\tableofcontents
%%%%%%%%%%%%%%%%%%%%%%%%%%%%%%%%%%%%%%%%%%
%%%%%%%%%%%%%%%%%%%%%%%%%%%%%%%%%%%%%%%%%%
\section{Introduction}
In the study of \textit{high-level neural networks}, Edelman et al. \cite{tononi1994measure} introduced a quantitative measure which is called \textit{neural complexity}. Neural complexity is used to capture the interplay between two fundamental aspects of brain partition organization. Other works like \cite{eickhoff2018imaging, fornito2016fundamentals} also give further investigations. In \cite{buzzi2012approximate, buzzi2012mean}, Buzzi et al. placed neural complexity within a natural mathematical function. They found that this function satisfies exchangeability and sub-additivity. Moreover, they presented a unified mathematical expression for these functions, which they called \textit{intricacy}.

Based on Buzzi's results and the definition of entropy, Peterson et al. \cite{petersen2018dynamical} further investigated the intricacy in dynamical systems. They introduced new \textit{intricacy} and \textit{average sample complexity} in topological dynamical systems, which are closely similar to entropy. They found that the maximum value of the average sample complexity of covers is equivalent to the topological entropy.

Li et al. \cite{li2021dynamical} extended the average sample complexity and intricacy to the case of countable amenable groups based on Peterson's work. They defined the \textit{intricacy} and \textit{the average sample complexity} for amenable group actions. Through their research, they stated that many properties of intricacy defined by Peterson still hold with the assistance of F{\o}lner sequences of amenable groups. They also established the relationship between intricacy and the classical topological entropy. Yang et al. \cite{yang2022dynamical} defined random intricacy and average sample complexity by using random covers, and conclude that their suprema over open random covers are equal to the random topological entropy. In \cite{xiao2024pressure}, Xiao et al. proposed the concept of \textit{the pressure of intricacy and the average sample complexity} for amenable group actions and demonstrated its variational principle for general dynamical systems.

The local entropy concept are mentioned both in topological and measure-theoretical cases, as discussed in \cite{ye2007entropy} and related results can be found in \cite{adler1965topological,blanchard1997variation,bowen1975equilibrium,danilenko2001entropy,keller1998equilibrium,walters1982introduction}.  Assume that $(X,T)$ is a topological system, $\mathcal{U}$ is a cover and $\mu$ is an invariant measure. Romagnoli \cite{romagnoli2018local} gave the definition of two kinds of \textit{local measure-theoreticial entropy} $h_\mu^+(T,\mathcal{U})$ and $h_\mu^-(T,\mathcal{U})$. He proved that for each invariant measure $\mu_0$,  $h_\mathrm{top}(T,\mathcal{U})=\max_{\mu}h_{\mu_0}^+(T,\mathcal{U})$. In \cite{glasner2003ergodic}, Glanser et al. showed that for each invariant measure $\mu$, $h_\mathrm{top}(T,\mathcal{U})=\max_{\mu}h_{\mu_0}^-(T,\mathcal{U})$. Huang et al. \cite{huang2004entropy} showed that if there exist a system $(X,T)$ and an invariant measure $\mu$ with $h_\mu^+(T,\mathcal{U})> h_\mu^-(T,\mathcal{U})$ then there is a uniquely ergodic $\mathbb{Z}$-systems with the same property. Therefore, for each invariant
measure $\mu$, $h_\mu^-(T,\mathcal{U})=h_\mu^+(T,\mathcal{U})$. 

Huang et al. \cite{HUANG_YE_ZHANG_2006} introduced two definitions of measure-theoretical conditional entropy for covers of $\mathbb{Z}$-actions in 2006. Furthermore, they extended these definitions to amenable groups in \cite{huang2011local} and proved the equivalence between two types of measure-theoretical entropy.

Li et al. raised an open question: How does the construction of intricacy and average sample complexity apply to other type notions? Inspired by this question, this paper investigate local intricacy and average sample complexity for amenable group actions. We note that these definitions are well-defined and many properties still hold when compared to the local entropy. Due to the relationship between intricacy and average sample complexity, this paper will mainly investigate the properties of average sample complexity and its variational principles.

The outline of this paper is as follow. In Section \ref{sec2}, we give a review of the basical concept and propositions of amenable groups and the orbital theory, which will be necessary for the subsequent sections. In Section \ref{sec3}, we introduce the definitions of the conditional measure theoretical intricacy and average sample complexity of a cover $\mathcal{U}$. That is
$$\mathrm{Asc} _\mu ^-(G,\mathcal{U}|Y ):=\lim_{n \to \infty} \frac{1}{|F_n|}\sum_{S\subseteq {F_n}}c_S^{F_n}\inf_{\alpha \in \mathcal{P}(X),\alpha \succeq \mathcal{U}  }H_\mu(\mathcal{\alpha}_S|Y )\text{ and }$$

$$\mathrm{Asc} _\mu ^+(G,\mathcal{U}|Y ):=\inf _{\alpha \in \mathcal{P}(X),\alpha \succeq \mathcal{U}  }\lim_{n \to \infty} \frac{1}{|F_n|}\sum_{S\subseteq {F_n}}c_S^{F_n}H_\mu(G,\alpha|Y ).  $$
We shall show some propositions about these conditional average sample complexity. In Section \ref{sec4}, we discuss the equivalence of two kinds of measure-theoretic average sample complexity both under $\mathbb{Z}$-actions and amenable group actions, i.e. for each cover $\mathcal{U}$ and invariant measure $\mu$ of $X$,
$
\mathrm{Asc} _\mu ^-(G,\mathcal{U})=\mathrm{Asc} _\mu ^+(G,\mathcal{U}).
$
In Section \ref{sec5}, we give the local variational principle for the average sample complexity.
\section{Preliminary}\label{sec2}
\subsection{The amenable group}
In this subsection, we shall recall some notations and theorems about amenable groups from Ornstein and Weiss \cite{ornstein1987entropy}, which we shall use in the following sections. The proofs of some theorems can also be located in Chapter 4 of \cite{kerr2016ergodic}.

Let $\mathcal{F}(G)$ be the family of all finite non-empty subsets of $G$ and $|*|$ the cardinality of the set.

\begin{definition}\label{def2.1}Let $G$ be a countable discrete infinite group. 
	\begin{itemize}
		\item[(1)] $G$ is called \textit{amenable} if for each $K\in \mathcal{F}(G)$ and $\delta>0$, there exists $F\in \mathcal{F}(G)$ such that $\frac{|F\Delta KF|}{|F|}<\delta.$
		\item[(2)] Let $F$, $A\in\mathcal{F}(G)$ and $\varepsilon>0$.
		\begin{itemize}
			\item[(i)]  We denote \textit{$F$-boundary of $A$} by $\partial_FA:=\{s\in G:Fs\cap A\neq\emptyset\,and\,Fs\cap (G\setminus A)\neq\emptyset\}.$
			\item[(ii)] $A$ is \textit{$(F,\varepsilon)$-invariant} if $|\partial_FA|\leq\varepsilon|A|$.
		\end{itemize}
		\item[(3)] A sequence $\{F_n\}_{n\in\mathbb{N}}\subseteq\mathcal{F}(G)$  is called a \textit{F{\o}lner sequence} if for every $F\in\mathcal{F}(G)$ and $\varepsilon>0$, there exists $N\in\mathbb{N}$ such that $\{F_n\}_{n\in\mathbb{N}}$ is $(F,\varepsilon)$-invariant for all $n\geq N$.
	\end{itemize}
\end{definition}

For a more detailed explanation of amenable groups, one can see \cite{kerr2016ergodic,li2021dynamical,ornstein1987entropy,weiss2001monotileable}. In this paper, the groups that we consider are always assumed to be  countable discrete infinite amenable groups.

The following definitions and theorems are due to Ornstein and Weiss \cite{ornstein1987entropy}. 

\begin{definition}\label{def2.5}
	Let $G$ be a countable amenable group and $A\in\mathcal{F}(G)$. We say that a collection $\{A_i\}_{i\in I}\subseteq\mathcal{F}(G)$ \textit{$\lambda$-covers} $A$ if
	$$ \frac{\left | \bigcup_{i\in I}A_i  \right |}{\left | A \right | } \ge \lambda .$$
	$\{A_i\}_{i\in I}$ is \textit{$\varepsilon$-disjoint} if there are pairwise disjoint sets $B_i\subset A_i$ such that for any $i\in I$,
	$$|B_i|\geq(1-\varepsilon)|A_i|.$$
	For $\delta \in[0,1)$, we say $\left \{ A_1,\dots ,A_k \right \}$ $\delta$\textit{- even covers A} if 
	\begin{itemize}
		\item[(i)] $A_i\subseteq A$ for $i=1,\dots,k$,
		\item[(ii)] there exists $M>0 $ such that 
		$\sum\limits_{i=1}^{k} 1_{A_i}(g)\le M \text{ and }\sum\limits_{i=1}^{k}\left | A_i \right | \ge (1-\delta )M\left | A \right |.$
	\end{itemize}
\end{definition}

\begin{definition}\label{def2.6}
	Let $A\in\mathcal{F}(G)$ and $\varepsilon>0$. A finite collection $\{F_1,F_2,\cdots,F_n\}\subseteq\mathcal{F}(G)$  is said to \textit{$\varepsilon$-quasitile} $A$, if there exist $C_1,C_2,\cdots,C_n\in\mathcal{F}(G)$ such that
	\begin{itemize}
		\item[(i)] $\bigcup^n_{k=1}F_kC_k\subset A$ , $F_iC_i\cap F_jC_j=\emptyset$ for $1\leq i<j\leq n$;
		\item[(ii)] $\{F_kC_k:k=1,\cdots,n\}$ $(1-\varepsilon)$-covers A;
		\item[(iii)] For $k=1,2,\cdots,n$, $\{F_kc:c\in C_k\}$ is $\varepsilon$-disjoint.
	\end{itemize}
	The subsets $C_1,C_2,\cdots,C_n$ are called the \textit{tiling centers}.
\end{definition}

\begin{definition}\label{def2.7}
	A set function $\varphi:\mathcal{F}(G)\rightarrow\mathbb{R}$ is
	\begin{itemize}
		\item[(a)] \textit{monotone} if $\varphi(E)\leq\varphi(F)$ for any $E,\,F\in\mathcal{F}(G)$ with $E\subset F$;
		\item[(b)] \textit{non-negative} if $\varphi (E)\ge 0$ for any $E\in\mathcal{F}(G)$;
		\item[(c)] \textit{$G$-invariant} if $\varphi(Eg)=\varphi(E)$ for any $g\in G$ and $E\in\mathcal{F}(G)$;
		\item[(d)] \textit{sub-additive} if $\varphi(E\cup F)\leq\varphi(E)+\varphi(F)$ for any $E\text{, }F\in\mathcal{F}(G)$ with $E\cap F=\emptyset$.
	\end{itemize}
\end{definition}

\begin{theorem}\label{ow}(Ornstein-Weiss \cite{ornstein1987entropy})
	Suppose that $G$ is an amenable group. Let $\varphi:\mathcal{F}(G)\rightarrow[0,\infty)$ be a monotone, non-negative, $G$-invariant and sub-additive function (for short m.n.i.s.a.). Then there exists $\lambda=\lambda(G,\varphi)\in[0,+\infty)$ depended only on $G$ and $\varphi$ such that
	$$\lim\limits_{n\to \infty}\frac{\varphi(F_n)}{|F_n|}=\lambda,$$
	for all F{\o}lner sequences $\{F_n\}_{n\in\mathbb{N}}$ of $G$.
	
\end{theorem}

The above limit theorem plays an important role in defining some dynamical invariants, such as the topological entropy and the intricacy of amenable groups in this paper.

\subsection{The neural complexity}

In \cite{tononi1994measure}, Tononi et al. introduced neural complexity in the context of high-level neural networks. To see what neural complexity means, we recall two classical definitions. 

\begin{definition}
	If $X$ and $Y$ are two random variables taking value in a finite space $E$ and $F$, respectively. We define \textit{the mutual information} between $X$ and $Y$ by
	$$
	I(X;Y):=H(X)+H(Y)-H(X,Y),
	$$
	where $$H(X)=-\sum_{x\in E}P(X=x)\log P(X=x),$$ $$H(X,Y)=-\sum_{x\in E, y\in F}P(X=x,Y=y)\log P(X=x,Y=y).$$
\end{definition}

Tononi et al.considered the system formed by a finite family of random variables $\mathcal{X} := (X_i)_{i\in E}$.
For any $S\subset E$, they divided the system into two separated parts $X_S=\{X_i:i\in S\}$, $X_{E\setminus S}:=\{X_i:i\in E\setminus S\}$.
\begin{definition}\label{def2.8}
	The \textit{neural complexity} is defined by
	$$
	\mathcal{I}(X_E):=\frac{1}{|E|+1}\sum_{S\subset E }\frac{1}{\binom{|E|}{|S|}}I(X_S;X_{E\setminus S}),
	$$
	where $I(X_\emptyset;X_{E})=I(X_{E};X_\emptyset)=0$.
\end{definition}

Tononi et al. \cite{tononi1994measure} used this notion to capture the interplay between the two fundamental aspects of brain organization.

Now we review the system of coefficients for $\mathbb{Z}$-actions and amenable group actions. One can refer to \cite{petersen2018dynamical,li2021dynamical} for details.
\begin{definition}
	Let $n^*=\{0,1,\cdots,n-1\}$ and $S=\{s_0,s_1,\cdots,s_{|S|-1}\}$ be a subset of $n^*$. A \textit{system of coefficients} for $\mathbb{Z}$-actions is defined to be a family of real numbers $\{c^n_S:S\subseteq n^*\}$ satisfying that $c^n_S\geq0$, $\sum_{S\subseteq n^*}c^n_S=1$ and $c^n_S=c^n_{X\setminus S}$, for all $n\in\mathbb{N}$ and $S\subseteq n^*$.
	
	More generally, if $G$ is a countable discrete group, a system of coefficients for $G$-actions is defined to be a family of numbers $\{c^A_S: S\subseteq A\}$ where $A\subseteq\mathcal{F}(G)$,  $c^A_S\geq0$, $\sum_{S\subseteq A}c^A_S=1$ and $c^A_S=c^A_{A\backslash S}$. 
\end{definition}

\begin{remark}
	\begin{itemize}
		\item[$\mathrm{(1)}$] $\left\{c^n_S:c^n_S=2^{-n},S\subseteq n^*\right\}$ and $\left\{c^A_S:c^A_S=2^{-|A|},{S\subseteq A}\right\}$ are the systems of coefficients. We call these two cases, respectively the uniform system of coefficients for $\mathbb{Z}$-actions and  $G$-actions.
		\item[$\mathrm{(2)}$] In this paper, we always assume that a system of coefficients for $G$-actions is \[c^A_S=\int_{[0,1]}x^{|S|}(1-x)^{|A|-|S|}\lambda(dx),\] where $\lambda$ is a symmetric probability measure on $[0,1]$ and some main theorems hold when $c^A_S=2^{-|A|}$.
		\item[$\mathrm{(3)}$] $\frac{1}{|A|+1}\frac{1}{\binom{|A|}{|S|}}$ is a special example of the system of coefficients.
	\end{itemize}
\end{remark}

\subsection{Orbital theory}
In this subsection, we shall introduce orbital theory, which can be referred to \cite{feldman1977ergodic,golodets1994classification} for more details. Let $(X,\mathcal{B}_X,\mu)$ be a Lebesgue space and $\mathcal{R}\subseteq X\times X$ a Borel equivalence relation on $X$. The group of $\mu$-preserving invertible transformations of $X$ is denoted by $Aut(X,\mu)$. We equip it with the (Polish) weak topology, which makes the following unitary representation continuous: \[Aut(X,\mu)\ni\gamma\mapsto U_\gamma\in\mathcal{U}(L^2(X,\mu)),\] where $U_\gamma f=f\circ\gamma^{-1}$ and $\mathcal{U}(L^2(X,\mu))$ is endowed with the strong operator topology.

Denote $\mathcal{R}(X):=\{y\in X:(x,y)\in \mathcal{R}\}$ for each $x\in X$. We say that $\mathcal{R}$ is \textit{of type I} if $|\mathcal{R}(x)|<+\infty$ for $\mu$-a.e. $x\in X$, that is, there exists $B\in\mathcal{B}_X$ such that $|B\cap\mathcal{R}(x)|=1$ for $\mu$-a.e. $x\in X$. We call $B$ a $\mathcal{R}$-\textit{fundamental domain}.

\begin{definition}
	Let $(X,\mathcal{B}_X,\mu)$ be a Lebesgue space and $\mathcal{R}\subseteq X\times X$ a Borel equivalence relation on $X$. We call $\mathcal{R}$ is
	\begin{itemize}
		\item[(a)] \textit{measure preserving} if it is generated by some countable sub-group $G\subseteq Aut(X,\mu)$ which is highly non-unique;
		\item[(b)] \textit{ergodic} if A belongs to the trivial  sub-$\sigma$-algebra of $\mathcal{B}_X$ when $A\in\mathcal{B}_X$ is $\mathcal{R}$-invariant (i.e.  $A=\bigcup_{x\in A}\mathcal{R}(X)$) ;
		\item[(c)] \textit{discrete} if $|\mathcal{R}(x)|\le |\mathbb{Z}|$ for $\mu$-a.e. $x\in X$;
		\item[(d)] \textit{hyperfinite} if there exists a sequence $R_1\subseteq R_2\subseteq \cdots$ of type I sub-relations of $\mathcal{R}$ with $\bigcup_{n\in\mathbb{N}} R_n = R$.
	\end{itemize}
\end{definition}

	Let $\mathcal{R}$ be measure preserving. We can observe that 
	$\mathcal{R}$ is countable if and only if it is conservative (i.e. $\mathcal{R}\cap (B\times B)\setminus \{(x,x):x\in X\}\neq \emptyset$ for each $B\in\mathcal{B}_X$ with $\mu(B)>0$ ). 
	 A measure preserving and discrete relation $\mathcal{R}$ is hyperfinite if and only if it is generated by a single transformation. 
	The orbit equivalence relation of a measure preserving action of a countable amenable group is hyperfinite.

The equivalence relations discussed in the following is measure preserving and discrete on $X$ where $(X,\mathcal{B}_X,\mu)$ is a Lebesgue space. Set
$$
[\mathcal{R}]:=\{\gamma\in Aut(X,\mu):(x,\gamma y)\in\mathcal{R}\text{ for } \mu\text{-}a.e.\text{ } x\in X\},
$$
$$
N[\mathcal{R}]:=\{\theta \in Aut(X,\mu):\theta\mathcal{R}(x)=\mathcal{R}(\theta x)\text{ for } \mu\text{-}a.e.\text{ } x\in X\}.
$$
\begin{definition}
	Let $A$ be a Polish group and $\gamma\in[\mathcal{R}]$. We call a Borel map $\phi:\mathcal{R}\to A$  cocycle if
	$\phi(x,z)=\phi(x,y)\phi(y,z)\text{ for all }(x,y),(y,z)\in\mathcal{R}$. Define
	\begin{itemize} 
		\item[(i)] A \textit{cocycle} $\phi\circ\theta$ by setting $\phi\circ\theta(x,y)=\phi(\theta x,\theta y)$ for all $(x,y)\in\mathcal{R}$;
		\item[(ii)] $\phi$-\textit{skew product extension} of $\mathcal{R}$ by
		$\mathcal{R}(\phi):=\{((x,y),(x',y'))\in (X\times Y)\times  (X\times Y):(x,x')\in\mathcal{R},y' = \phi(x, x')y\};$
		\item[(iii)] $\gamma_\phi$ by setting $\gamma(x,y)=(\gamma x,\phi(\gamma x,x)y)\text{ for each }(x, y)\in X\times Y$. $\gamma_\phi$ is called $\phi$-skew product extension of $\gamma$.
	\end{itemize}
\end{definition}

It is easy to show that $\mathcal{R}(\phi)$ is a measure preserving discrete equivalence relation on $(X\times Y, \mathcal{B}_X\times\mathcal{B}_Y,\mu\times\nu)$.

\subsection{$G$-systems}
\begin{definition}\label{def3.1}
	A pair $(X, G)$ is called a \textit{$G$-system} if $X$ is a compact metric space and $\psi:G\times X\rightarrow X$ is a continuous map with the following conditions:
	\begin{itemize}
		\item[(i)] $\psi(e_G,x)=x$ for every $x\in X$ and the identity $e_G\in G$;
		\item[(ii)] $\psi(g_1,\psi(g_2,x))=\psi(g_1g_2,x)$ for every $g_1,\, g_2\in G$ and $x\in X$.
	\end{itemize}
	In general, we write $gx:=\psi(g,x)$. Let $T:X\to X$ and $R:Y\to Y$ be two invertible measure-preserving transformations. When $G=\mathbb{Z}$, set $\psi(i,x)=T^ix$ for each $i\in\mathbb{Z}$ and $x\in X$. We denote $(X,G)$ and $(Y,G)$ by $(X,T)$ and $(Y,R)$. We call $(X,T)$ and $(Y,R)$ $\mathbb{Z}$\textit{-systems}.
\end{definition}

Let $(X, G)$ be a $G$-system. A \textit{cover} of $X$ is a family of Borel subsets of $X$ whose union is $X$. An \textit{open cover} of $X$ is a cover that consists of open sets. A \textit{partition} of $X$ is a cover of $X$ consisting of pairwise disjoint sets. Denote the set of finite covers by $\mathcal{C}_X$, the set of finite open covers by $\mathcal{C}^o_X$, and the set of finite partitions by $\mathcal{P}_X$. Given two covers of $\mathcal{U},\,\mathcal{V}\in\mathcal{C}_X$, $\mathcal{U}$ is said to be \textit{finer} than $\mathcal{V}$ (write $\mathcal{U}\succeq\mathcal{V}$) if each member of $\mathcal{U}$ is contained in some member of  $\mathcal{V}$. Let $\mathcal{U}\vee\mathcal{V}=\{U\cap V:U\in\mathcal{U},\,V\in\mathcal{V}\}$. Given $S\in\mathcal{F}(G)$ and $\mathcal{U}\in\mathcal{C}_X$, set $\mathcal{U}_S=\bigvee\limits_{g\in S}g^{-1}\mathcal{U}$ . 

Let $(X,G)$ and $(Y,G)$ be two $G$-systems. A onto map $\varphi :(X,G)\to (Y,G)$ is called a \textit{factor map} if $\varphi$ is a continuous and satisfies $\varphi g=g\varphi $ for $g\in G$. When $G=\mathbb{Z}$, we note that a onto map $\varphi :(X,T)\to (Y,R)$ is a factor map if $\varphi$ is a continuous and satisfies $\varphi R=T\varphi $.

\section{Topological and measure-theoretic intricacy }\label{sec3}
\subsection{Topological intricacy }

Let $\varphi :(X,G)\to (Y,G)$ be a factor map between two $G$-systems. For a subset $K\subseteq X$ and $\mathcal{U}$, $\mathcal{W}\in \mathcal{C} _X$, set $$
N\left ( \mathcal{U}|K  \right ) =\min\left \{ |\mathcal{W}|:\mathcal{W}\subseteq \mathcal{U}\text{ and }K\subseteq \bigcup_{W\in\mathcal{W}}W  \right \}\text{ and }N\left ( \mathcal{U} |\mathcal{W} \right ) =\sup_{W\in\mathcal{W}  } N\left ( \mathcal{U} |W \right )$$ $\left ( \text{letting } N(\mathcal{U}):=N(\mathcal{U}|\{X\}\right) )$.
For $y\in Y$, we define 
$$N(\mathcal{U}|y)=N(\mathcal{U}|\varphi^{-1} (y))\text{ and }N(\mathcal{U}|Y)=\sup_{y\in Y} N(\mathcal{U}|y).$$

%我们将intricacy函数中的$\frac{1}{|A|+1}\frac{1}{\biom{|A|}{|S|}}$替换成$c^A_S$，并且互信息函数替换为$ \log\left( \frac{N(\mathcal{U}_{S}|Y)N(\mathcal{U}_{F_n\setminus S}|Y)}{N(\mathcal{U}_{F_n}|Y)} \right)\text{and }$.
We replace the term $\frac{1}{|E|+1}\frac{1}{\binom{|E|}{|S|}}$ in Definition \ref{def2.8} with $c^E_S$, and substitute the mutual information function with $\log N(\mathcal{U}_{S}|Y)+\log N(\mathcal{U}_{E\setminus S}|Y)-\log N(\mathcal{U}_{E}|Y)$. Then we introduce the following definitions.
\begin{definition}\label{def3.1}
	Let $\left \{ F_n \right \} _{n\in\mathbb{N} }$ be a F\o lner sequence of $G$. For $\mathcal{U}\in \mathcal{C}_X$ and $F\in \mathcal{F}(G)$, we define the \textit{conditional topological intricacy} and \textit{average sample complexity} of $\mathcal{U}$ with respect to $(Y,G)$, respectively, by
	$$\mathrm{Int}_{\mathrm{top} }\left ( G,\mathcal{U}|Y \right ) :=\lim_{n \to \infty} \frac{1}{\left | F_n \right | }\sum_{S\subseteq F_n}c^{F_n}_S \log\left( \frac{N(\mathcal{U}_{S}|Y)N(\mathcal{U}_{F_n\setminus S}|Y)}{N(\mathcal{U}_{F_n}|Y)} \right),$$
	$$\mathrm{Asc}_{\mathrm{top} }\left ( G,\mathcal{U}|Y \right ) :=\lim_{n \to \infty} \frac{1}{\left | F_n \right | }\sum_{S\subseteq F_n}c^{F_n}_S \log N(\mathcal{U}_{S}|Y)  ,$$
	where $\mathrm{Int}_{\mathrm{top} }\left ( G,\mathcal{U} |Y \right )$ and $\mathrm{Asc}_{\mathrm{top} }\left ( G,\mathcal{U} |Y \right )$ are independent of the choice of F{\o}lner sequences.
	The \textit{conditional intricacy} and \text{average sample complexity} of $(X,G)$ with respect to $(Y,G)$ are defined, respectively, by
	$$\mathrm{Int}_{\mathrm{top} }\left ( G,X|Y  \right ) :=\sup _{\mathcal{U}\in\mathcal{C}_X^o  } \mathrm{Int}_{\mathrm{top} }\left ( G,\mathcal{U}|Y  \right )\text{, }\mathrm{Asc}_{\mathrm{top} }\left ( G,X|Y  \right ) :=\sup _{\mathcal{U}\in\mathcal{C}_X^o  } \mathrm{Asc}_{\mathrm{top} }\left ( G,\mathcal{U}|Y  \right ).$$
	
\end{definition}

\begin{remark}\label{rmk3.1}$\empty$
	\begin{itemize}
		\item[$\mathrm{(1)}$] Similar to the proof of Theorem 3.4 in \cite{xiao2024pressure}, we can observe that $F_n \mapsto \sum\limits_{S \subseteq F_n} c^{F_n}_S \log N(\mathcal{U}_S | Y)$ satisfies the conditions of Theorem \ref{ow}. Therefore, the definitions are well-defined.
		\item[$\mathrm{(2)}$] When $G=\mathbb{Z}$, we can also define the conditional topological intricacy and average sample complexity of $\mathcal{U}$ with respect to $(Y,R)$, respectively, by
		
		$$\mathrm{Int}_{\mathrm{top} }\left ( T,\mathcal{U}|Y \right ) :=\lim_{n \to \infty} \frac{1}{n }\sum\limits_{S\subseteq n^*}c^{n}_S\log\left( \frac{N(\mathcal{U}_{S}|Y)N(\mathcal{U}_{n^*\setminus S}|Y)}{N(\mathcal{U}_{n^*}|Y)} \right) ,$$
		
		$$\mathrm{Asc}_{\mathrm{top} }\left ( T,\mathcal{U}|Y \right ) :=\lim_{n \to \infty} \frac{1}{n }\sum_{S\subseteq n^*}c^{n}_S\log  N(\mathcal{U}_{S}|Y) . $$
		
		The conditional topological intricacy and average sample complexity of $(X,T)$ with respect to $(Y,R)$ are defined, respectively, by
		$$\mathrm{Int}_{\mathrm{top} }\left (T,X|Y  \right ) :=\sup\limits _{\mathcal{U}\in\mathcal{C}_X^o  } \mathrm{Int}_{\mathrm{top} }\left ( T,\mathcal{U}|Y  \right ), \mathrm{Asc}_{\mathrm{top} }\left (T,X|Y  \right ) :=\sup\limits _{\mathcal{U}\in\mathcal{C}_X^o  } \mathrm{Asc}_{\mathrm{top} }\left ( T,\mathcal{U}|Y  \right ).$$
		\item[$\mathrm{(3)}$] It is easy to obtain that 
		$\mathrm{Int}_{\mathrm{top} }\left ( G,\mathcal{U}|Y \right )=2\mathrm{Asc}_{\mathrm{top} }\left ( G,\mathcal{U}|Y \right )-h_\mathrm{top}\left ( G,\mathcal{U}|Y \right )$, where $h_\mathrm{top}\left ( G,\mathcal{U}|Y \right )=\lim\limits_{n\to\infty}\frac{1}{|F_n|}\log N(\mathcal{U}_{F_n}|Y)$.
		Therefore, for the topological case, the following sections will focus only on the average sample complexity.
		\item[$\mathrm{(4)}$] If $(Y,G)$ is a trivial system, the conditional average sample complexity is the same as the average sample complexity in \cite{petersen2018dynamical}. That is, $\mathrm{Asc}_{\mathrm{top} }\left ( G,\mathcal{U}|\{\emptyset, Y\}  \right ) =\mathrm{Asc}_{\mathrm{top} }\left ( G,\mathcal{U} \right ) $ and $\mathrm{Asc}_{\mathrm{top} }\left ( G,X|\{\emptyset, Y\}  \right ) =\mathrm{Asc}_{\mathrm{top} }\left ( G,X \right )$.
		
		%改成example 直接把例子单独写出来
	\end{itemize}
\end{remark}

\subsection{Measure-theoretic intricacy }

Denote the set of all Borel probability measures on $X$ by $\mathcal{M}(X)$, the set of all $G$-invariant Borel probability measures on $X$ by $\mathcal{M}(X,G)$ and the set of all ergodic measures by $\mathcal{M}^e(X,G)$. Note that $\mathcal{M}(X,G)\ne \emptyset$ when $G$ is an amenable group, and both $\mathcal{M}(X)$ and $\mathcal{M}(X,G)$ are convex compact metric spaces when they are endowed with weak $^*$-topology.

Let $(X,G)$ and $(Y,G)$ be two $G$-systems and $\pi :(X,G)\to(Y,G) $ is a factor map between two $G$-systems. Given 
$\alpha \in \mathcal{P} _X$, define 

$$H_{\mu }(\alpha |Y)=\sum_{A\in \alpha }\int_{X}-\mathbb{E} (1_A|\pi^{-1}(\mathcal{B}(Y)))\log  \mathbb{E} (1_A|\pi^{-1}(\mathcal{B}(Y)))d\mu, $$
where $\mathbb{E} (1_A|\pi^{-1}(\mathcal{B}(Y)))$ is the expectation of $1_A$ with respect to $\pi^{-1}(\mathcal{B}(Y))$. It is not hard to show that $H_{\mu }(\alpha |\pi^{-1}(\mathcal{B}(Y)))$ increases with respect to $\alpha$ and decreases with respect to $\pi^{-1}(\mathcal{B}(Y))$.

\begin{definition}\label{def3.2}
	Let $\mu \in \mathcal{M}(X,G) $. For $\alpha\in\mathcal{P}_X$, we define  the \textit{conditional measure-theoretic intricacy} and \textit{average sample complexity} of $\alpha$ with respect to $(Y,G)$, respectively, by
	$$\mathrm{Int}_{\mu }\left ( G,\alpha|Y \right ) :=\lim_{n \to \infty} \frac{1}{\left | F_n \right | }\sum_{S\subseteq F_n}c^{F_n}_S \left[H_\mu(\alpha_{S}|Y)+H_\mu(\alpha_{F_n\setminus S}|Y)-H_\mu(\alpha_{F_n}|Y) \right],$$
	$$\mathrm{Asc}_{\mu } (G,\alpha|Y ):=\lim_{n \to \infty} \frac{1}{|F_n|} \sum_{S\subseteq F_n}c^{F_n}_SH_\mu (\alpha _S|Y)\left (= \inf _{n\in \mathbb{N} }\frac{1}{|F_n|}  \sum_{S\subseteq F_n}c^{F_n}_SH_\mu (\alpha _S|Y) \right ).$$
	where $\mathrm{Int}_{\mu}\left ( G,\alpha |Y \right )$ and $\mathrm{Asc}_{\mu }\left ( G,\alpha |Y \right )$ are independent of the choice of F{\o}lner sequences.
	
	The conditional measure-theoretic intricacy and average sample complexity for $(X,G)$ with respect to $(Y,G)$ can be defined, respectively, by 
	$$\mathrm{Int}_{\mu } (G,X|Y):=\sup _{\alpha \in \mathcal{P}(X) }\mathrm{Int}_{\mu } (G,\alpha|Y )\text{, }\mathrm{Asc}_{\mu } (G,X|Y):=\sup _{\alpha \in \mathcal{P}(X) }\mathrm{Asc}_{\mu } (G,\alpha|Y ).$$
	
\end{definition}

\begin{remark}\label{rmk3.2}$\empty$
	\begin{itemize}
		\item[$\mathrm{(1)}$] Similar to the proof of \cite[Theorem 3.4]{xiao2024pressure}, we can check that the definition is well-defined.
		\item[$\mathrm{(2)}$] If $(Y,G)$ is a trivial system, the conditional measure-theoretical average sample complexity is the same as the definition in \cite{petersen2018dynamical}. That is, $\mathrm{Asc}_{\mu}\left ( G,\alpha|\{\emptyset, Y\}  \right ) =\mathrm{Asc}_{\mu }\left ( G,\alpha \right ) $ and $\mathrm{Asc}_{\mu}\left ( G,X|\{\emptyset, Y\}  \right ) =\mathrm{Asc}_{\mathrm{\mu} }\left ( G,X \right ) $ .
		\item[$\mathrm{(3)}$] Similar to the topological case, the following sections will also focus only on the measure-theoretic average sample complexity.
		\item[$\mathrm{(4)}$] We can find $\mu\in\mathcal{M}(X)$ and $\mathcal{U}\in\mathcal{C}_X$ such that $\mathrm{Asc}_\mu(T,\mathcal{U})<h_\mu(T,\mathcal{U})$. One can refer to \cite[Example 4.5]{petersen2018dynamical} for details.
		
	\end{itemize}
\end{remark}
Following the idea proposed by Romagnoli \cite{danilenko2001entropy}, we now introduce a new concept of the conditional $\mu$-measure-theoretical average sample complexity for covers. 

Let $\pi: (X,G) \to (Y,G)$ be a factor map between two $G$-systems and $\mu\in\mathcal{M}(X)$. For $\mathcal{U} \in \mathcal{C}_X$, we define
$H_\mu (\mathcal{U}|Y)=\inf\limits _{\alpha \in \mathcal{P}_X,\alpha \succeq \mathcal{U}}H_\mu(\alpha |\pi ^{-1}\mathcal{B}(Y)). $ The following lemma can be referred to \cite[Lemma 2.2]{HUANG_YE_ZHANG_2006} for a detailed proof.
%huang2011文献
\begin{lemma}\label{lem3.1}  If $\mathcal{U},\mathcal{V}\in\mathcal{C}_X$, then the following holds:
	
	\begin{itemize}
		\item[$\mathrm{(1)}$] $0\le H_\mu(\mathcal{U}|Y )\le \log N(\mathcal{U}|Y );$
		
		\item[$\mathrm{(2)}$] If $\mathcal{U}\succeq \mathcal{V}$ ,then $H_\mu (\mathcal{U}|Y)\ge H_\mu (\mathcal{V}|Y);$
		\item[$\mathrm{(3)}$] $\mathrm{Asc} _\mathrm{top} (\mathcal{U}\vee \mathcal{V} |Y )\le\mathrm{Asc} _\mathrm{top} (\mathcal{U}|Y )+\mathrm{Asc} _\mathrm{top} (\mathcal{V} |Y )$;
		
		\item[$\mathrm{(4)}$] $\mathrm{Asc} _\mathrm{top} (g^{-1}\mathcal{U} |Y )=\mathrm{Asc} _\mathrm{top} (\mathcal{U} |Y ) \text{ for all }g\in G.$
	\end{itemize}

\end{lemma}

\begin{definition}\label{def3.3}
	Let $\mu\in\mathcal{M}(X,G)$. Given $\mathcal{U}\in\mathcal{C}_X$, we define the $\mu^-$\textit{-conditional average sample complexity} of $\mathcal{U}$ with respect to $(Y,G)$  and $\mu^+$\textit{-conditional average sample complexity} of $\mathcal{U}$ with respect to $(Y,G)$, respectively, by
	$$\mathrm{Asc} _\mu ^-(G,\mathcal{U}|Y ):=\lim_{n \to \infty} \frac{1}{|F_n|}\sum_{S\subseteq {F_n}}c^{F_n}_SH_\mu(\mathcal{U}_S|Y )\left(=\inf _{n\ge 1}\frac{1}{|F_n|}\sum_{S\subseteq F_n}c^n_SH_\mu(\mathcal{U}_S|Y )\right),$$
	
	$$\mathrm{Asc} _\mu ^+(G,\mathcal{U}|Y ):=\inf _{\alpha \in \mathcal{P}(X),\alpha \succeq \mathcal{U}  }\mathrm{Asc}_\mu(G,\alpha|Y ).  $$
	where $\mathrm{Asc}_{\mu }^-\left ( G,\mathcal{U} |Y \right )$ and $\mathrm{Asc}_{\mu }^+\left ( G,\mathcal{U} |Y \right )$ are independent of the choice of F{\o}lner sequences.
	
	We define $\mu^-$\textit{-conditional average sample complexity} of $(X,G)$ with respect to $(Y,G)$ and $\mu^+$\textit{-conditional average sample complexity} of $(X,G)$ with respect to $(Y,G)$, respectively, by
	$$\mathrm{Asc} _\mu ^-(G,X|Y ):=\sup_{\mathcal{U}\in\mathcal{C}_X}\mathrm{Asc} _\mu ^-(G,\mathcal{U}|Y )\text{ , }\mathrm{Asc} _\mu ^+(G,X|Y ):=\sup_{\mathcal{U}\in\mathcal{C}_X}\mathrm{Asc} _\mu ^+(G,\mathcal{U}|Y ).$$
	
\end{definition}

In this subsection, we aim to present some propositions with regard to two types of conditional local average sample complexity of covers.
\begin{proposition}\label{prop3.2}
	Let $\pi:(X,T)\to(Y,R)$ and $\varphi:(Z,S)\to(X,T)$ be two factor maps. If $\nu\in\mathcal{M}(Z,S)$, $\mu=\varphi\nu$ and $\mathcal{U}\in\mathcal{C}_X$, then 
	\begin{itemize}
		\item[$\mathrm{(1)}$] $\mathrm{Asc}_\mu ^-(T,\mathcal{U}|Y )\le \mathrm{Asc}_\mu ^+(T,\mathcal{U}|Y )$;
		
		\item[$\mathrm{(2)}$] $\mathrm{Asc}_\mathrm{top}  (S,\varphi ^{-1}\mathcal{U} |Y)=\mathrm{Asc}_\mathrm{top}  (T,\mathcal{U} |Y)$;
		\item[$\mathrm{(3)}$] $\mathrm{Asc}_\nu  (S,\varphi^{-1}\mathcal{U} |Y)=\mathrm{Asc}_\mu  (T, \mathcal{U} |Y)$;  \end{itemize}
		If the systems of coefficients in the following $\mathrm{Asc}_\mu ^-(T,\mathcal{U}|Y )$, $ \mathrm{Asc}_\mu ^+(T^n,\mathcal{U}_S|Y )$ and $\mathrm{Asc}_{\mathrm{top}}(T,\mathcal{U}|Y )$ are uniform, then we have
		\begin{itemize}\item[$\mathrm{(4)}$]  $\mathrm{Asc}_\mu ^-(T,\mathcal{U}|Y )=\lim\limits_{n\to\infty}\frac{1}{n}\sum\limits_{0\in S\subseteq n^*}\frac{1}{2^{n-1}}H_\mu(\mathcal{U}_S|Y)$;
		\item[$\mathrm{(5)}$] $\mathrm{Asc}_\mu ^-(T,\mathcal{U}|Y )=\frac{1}{M}\sum\limits_{0\in B\subseteq M^*} \frac{1}{2^{M-1}}\mathrm{Asc}_\mu ^-(T^M,\mathcal{U}_B|Y )$ for each $M\in \mathbb{N}$;
		\item[$\mathrm{(6)}$]   $\mathrm{Asc}_\mu ^-(T,\mathcal{U}|Y )=\lim\limits_{n \to \infty} \frac{1}{n}\sum\limits_{0\in S\subseteq n^*}\frac{1}{2^{n-1}} \mathrm{Asc}_\mu ^+(T^n,\mathcal{U}_S|Y )$;
		\item[$\mathrm{(7)}$]  $\mathrm{Asc}_{\mathrm{top}}(T,\mathcal{U}|Y )=\frac{1}{M}\sum\limits_{0\in B\subseteq M^*} \frac{1}{2^{M-1}}\mathrm{Asc}_{\mathrm{top}}(T^M,\mathcal{U}_B|Y )$ for each $M\in \mathbb{N}$.
	\end{itemize}
	
\end{proposition}
\begin{proof}
	(1), (2) and (3) are easy to prove. The proof of (4) can be refer to \cite[Proposition 2.5]{petersen2018dynamical}. So we shall prove (5) and (6). By the definition of $\mathrm{Asc}_\mu ^-(T,\mathcal{U}|Y )$, we have
	%验证（2）是不是对的
	$$
	\begin{aligned}
		&\mathrm{Asc}_\mu ^-(T,\mathcal{U}|Y )=\lim_{N\to\infty}\frac{1}{MN}\sum_{0\in L\subseteq (MN)^*}\frac{1}{2^{MN-1}}H_\mu(\mathcal{U}_L|Y)\\
		=&\frac{1}{M}\lim_{N\to\infty}\frac{1}{N}\frac{1}{2^{MN-1}}\sum_{0\in S\subseteq N^*}\sum_{0\in B\subseteq M^*}\frac{2^{MN-1}}{2^{M+N-2}}H_\mu(\mathcal{U}_{SM+B}|Y)=\frac{1}{M}\sum_{0\in B\subseteq M^*} \frac{1}{2^{M-1}}\mathrm{Asc}_\mu ^-(T^M,\mathcal{U}_B|Y ).
	\end{aligned}
	$$
	We prove (6). For any $M\in\mathbb{N}$,
	$$
	\begin{aligned}
		\mathrm{Asc}_\mu ^-(T,\mathcal{U}|Y )&=\frac{1}{M}\sum_{0\in B\subseteq M^*} \frac{1}{2^{M-1}}\mathrm{Asc}_\mu ^-(T^M,\mathcal{U}_B|Y )\le \frac{1}{M}\sum_{0\in B\subseteq M^*} \frac{1}{2^{M-1}}\mathrm{Asc}_\mu ^+(T^M,\mathcal{U}_B|Y ) \\
		&\le \frac{1}{M}\sum_{0\in B\subseteq M^*} \frac{1}{2^{M-1}}H_\mu(\mathcal{U}_B|Y).
	\end{aligned} 
	$$
	Letting $M\to\infty$, we have $\mathrm{Asc}_\mu ^-(T,\mathcal{U}|Y )=\lim\limits_{M \to \infty} \frac{1}{M}\sum\limits_{0\in B\subseteq M^*}\frac{1}{2^{M-1}} \mathrm{Asc}_\mu ^+(T^M,\mathcal{U}_B|Y )$. 
	
	The proof of (7) is similar to (4).
	
\end{proof}

\begin{proposition}\label{equation}
	Let $(X,G)$ and $(Y,G)$ be two $G$-systems. Assume that $d>0$ is a fixed integrate number, for each $\varepsilon>0$, there exists $\delta>0$ such that if $\mathcal{U}=\{U_1,\dots,U_d\}$ and $\mathcal{V}=\{V_1,\dots,V_d\}$ are any two covers of $X$ with $\sum\limits_{i=1}^d\mu(U_i\bigtriangleup V_i)<\delta$ then 
	$$
	\left|  \mathrm{Asc}_\mu ^+(G,\mathcal{V}|Y)-\mathrm{Asc}_\mu ^+(G,\mathcal{U}|Y)\right|<\varepsilon.
	$$
\end{proposition}
\begin{proof}
	The proof of this lemma is similar to \cite[Theorem 4.15]{walters1982introduction}.
\end{proof}

The following lemma is used to prove theorems in Subsection 3.4.
The proof of the following lemma can be refer to \cite{huang2004entropy}.
\begin{lemma}\label{lem3.2}
	Let $(X,G)$ be a $G$-system, $\mu\in\mathcal{M}(X)$ and $\alpha\in\mathcal{P}(X)$. If $\{\mathcal{A}_n\}_{n\in\mathbb{N}}$ is an increasing sequence of measurable sub-$\sigma$-algebras with $\mathcal{A}_n\nearrow\mathcal{A}$, then $H_\mu (\alpha |\mathcal{A}_n)\nearrow  H_\mu (\alpha |\mathcal{A})$.
\end{lemma}

\begin{lemma}\label{lem3.3}
	Let $\varphi :(X,G)\to (Y,G)$ be a factor map between two $G$-systems and $\alpha\in\mathcal{P}(X)$. Then the following holds:
	\begin{itemize}
		\item[$\mathrm{(1)}$] The function $H_* (\alpha|Y)$ is concave on $\mathcal{M}(X)$;
		
		\item[$\mathrm{(2)}$]  The function $\mathrm{Asc} _* (G,\alpha|Y) $ and $\mathrm{Asc} _* (G,X|Y) $ are affine on $\mathcal{M}(X,G)$.
		
	\end{itemize}
\end{lemma}
\begin{proof}
	The proof of (1) can be refer to \cite{HUANG_YE_ZHANG_2006} and the proof of (2) is similar to \cite{HUANG_YE_ZHANG_2006}.
\end{proof}
Let $X$ be a compact metric space. A real-valued function $f:X\to\mathbb{R}$ is called \textit{upper semi-continuous} (for short u.s.c.), if one of the following conditions holds:

\begin{itemize}
	\item[(i)] $\limsup\limits_{x'\to x}f(x')\le f(x)$ for each $x\in X$.
	
	\item[(ii)]  The set $\{x\in X:f(x)\ge a\}$ is closed for each $a\in\mathbb{R}$.
	
\end{itemize}

\begin{proposition}\label{thm3.1}
	Let $\pi:(X,G)\to(Y,G)$ be a factor map between two $G$-systems and $\mathcal{U}\in \mathcal{C}^o_X $. Then $\mathrm{Asc} _* ^+(G,\mathcal{U}|Y):\mathcal{M} (X,G)\to \mathbb{R} _+$ and $\mathrm{Asc} _* ^-(G,\mathcal{U}|Y):\mathcal{M} (X,G)\to \mathbb{R} _+$ are  u.s.c. on $\mathcal{M} (X,G)$.
\end{proposition}
\begin{proof}
	The  proof of this proposition is similar to the case of local entropy in \cite{huang2011local}.
\end{proof}

\subsection{The relationship between topological and measure-theoretical average sample complexity }
In this subsection, we shall present the proofs of the variational principle for two kinds of conditional average sample complexity of covers.

The following lemma plays an important role in the proof of conditional average sample complexity, which can be referred to \cite{HUANG_YE_ZHANG_2006}.
\begin{lemma}\label{lem3.4}
	Let $\mathcal{U}\in\mathcal{C}^o_X$ and $\alpha_l\in\mathcal{P}_X$ with $\alpha_l\succeq \mathcal{U}$, $1\le l\le K$. Then for each $S\in\mathcal{F}(G)$ there exists a finite subset $B_S\subseteq X$ such that each element of $(\alpha_l)_S$ contains at most one point of $B_S$, $l=1,\dots,K$ and $|B_S|\ge \frac{N(\mathcal{U}_S)}{|K|}$.
\end{lemma}
\begin{theorem}\label{thm3.2}
	Let $\pi:(X,G)\to(Y,G)$ be a factor map between two $G$-systems, $A\subseteq G$ , $\{c_S^A=\frac{1}{2^{|A|}}:S\subseteq A\}$ and $\mathcal{U}\in \mathcal{C}^o_X $. Then there exists $\mu\in\mathcal{M}(X,G)$ such that $$\mathrm{Asc}_\mathrm{top} (G,\mathcal{U}|Y)\le \mathrm{Asc}_\mu^+ (G,\mathcal{U}|Y).$$
\end{theorem}
\begin{proof}
	$\mathbf{Step 1:}$ Assume that $X$ is a zero-dimensional space.
	
	Let $\mathcal{U}=\{U_1,\dots,U_d\}$. We set $\mathcal{U}^*=\{\alpha=\{A_1,\dots,A_d\}\in\mathcal{P}(X): A_i\subseteq U_i, i\in\{1,\dots,d\}\}$ and $\{\alpha_l\}_{l\in\mathbb{N}}\subseteq\mathcal{U}^*$ an enumeration of the family of partitions which consists of clopen subsets. Then it is easy to show that $\mathrm{Asc}_\nu ^+(G,\mathcal{U})=\inf\limits_{l\in\mathbb{N}}\mathrm{Asc}_\nu(G,\alpha_l)$ for each $\nu\in\mathcal{M}(X,G)$.
	
	Since $G$ is infinite, let $\{F_n\}_{n\in\mathbb{N}}$ be a F{\o}ner sequence of $G$ with $|F_n|\ge n$. By Lemma \ref{lem3.4}, for each $n\in\mathbb{N}$ and $S\subseteq F_n$ there exists a finite subset $B_{{F_n},S}$ such that 
	$
	|B_{{F_n},S}|\ge \frac{N(\mathcal{U}_S)}{|K|},
	$
	and each element of $(\alpha_l)_S$ contains most one point of $S\subseteq F_n$, for $l=1,\dots,n$. Let
	$$
	\nu_{S,n}=\frac{1}{|B_{{F_n},S}|}\sum_{x\in B_{{F_n},S}}\delta _x\text{ and }\mu_{n}=\frac{1}{|F_n|}\sum_{S\subseteq F_n}\frac{1}{2^{|F_n|-1}}\sum_{g\in S}g\nu_{S,n}.
	$$
	It is easy to show that $\nu_{S,n}\in \mathcal{M}(X)$ for each $n\in\mathbb{N}$ and $S\subseteq F_n$.  We choose a sub-sequence $\{n_j\}_{j\in\mathbb{N}}$ such that $\mu_{n_j}\to\mu$ in weak$^*$-topology of $\mathcal{M}(X)$ as $j\to\infty$. It is not hard to check that $\mu\in\mathcal{M}(X,G)$. Now we shall prove that $\mu$ satisfies $ \mathrm{Asc}_{top} (G,\mathcal{U})\le\mathrm{Asc}_\mu ^+(G,\mathcal{U})$.
	
	Fix an $l\in\mathbb{N}$ and each $n\ge l$. By the construction of $B_{{F_n},S}$, we know that 
	\begin{equation}\label{eq5.1}
		\begin{aligned}
			\log N(\mathcal{U}_S)-\log n \le \log |B_{{F_n},S}|&=\sum_{x\in B_{{F_n},S}}-\nu_{S,n}(\{x\})\log \nu_{S,n}(\{x\})=H_{\nu_{S,n}}((\alpha_l)_S).
		\end{aligned}
	\end{equation}
	For each $B\in\mathcal{F}(G)$, we have 
	\begin{equation}\label{eq5.2}
		\begin{aligned}
			&\frac{1}{|F_n|}\sum _{S\subseteq F_n} \frac{1}{2^{|F_n|}}H_{\nu_{S,n}}((\alpha_l)_S)\le \frac{1}{|F_n|}\sum _{S\subseteq F_n} \frac{1}{2^{|F_n|}}\sum_{g\in S}\frac{1}{|B|}\sum_{A\subseteq B}c^B_AH_{\nu_{S,n}}((\alpha_l)_{Ag})\\+& \frac{1}{|F_n|}\sum _{S\subseteq F_n} \frac{1}{2^{|F_n|}}\sum_{g\in S}\frac{1}{|B|}\sum_{A\subseteq B} c^B_A\left|S\setminus \{g\in G:A^{-1}g\subseteq S\}\right| \cdot \log |\alpha_l|\\
			\le &\frac{1}{|F_n|\cdot|B|}\sum _{S\subseteq F_n}\frac{1}{2^{|F_n|}}\sum_{g\in S}\sum_{A\subseteq B} c^B_AH_{g\nu_{S,n}}((\alpha_l)_{A})+\frac{1}{|B|}\cdot \log d\le \frac{1}{|B|}\sum_{A\subseteq B}c^B_AH_{\mu_n}((\alpha_l)_{A})+\frac{1}{|B|}\cdot \log d.
		\end{aligned}
	\end{equation}
	By \eqref{eq5.1} and \eqref{eq5.2}, we obtain
	\begin{equation}\label{eq5.3}
		\begin{aligned}
			\frac{1}{|F_n|}\sum _{S\subseteq F_n}c^{F_n}_S\log N(\mathcal{U}_S)&\le\frac{1}{|B|}\sum_{A\subseteq B} c^B_AH_{\mu_n}((\alpha_l)_{A})+\frac{\log n}{|F_n|}+\frac{1}{|B|}\cdot \log d.
		\end{aligned}
	\end{equation}
	Noting that $\lim\limits_{j\to\infty}H_{\mu_{n_j}}((\alpha_l)_A)=H_{\mu}((\alpha_l)_A)$, by substituting $n$ with $n_j$ in \eqref{eq5.3}, we have 
	$$
	\begin{aligned}
		\mathrm{Asc}_{top} (G,\mathcal{U})\le\frac{1}{|B|}\sum_{A\subseteq B} c^B_AH_{\mu}((\alpha_l)_{A})+\frac{1}{|B|}\cdot \log d.
	\end{aligned}
	$$
	Replacing $B$ by $F_m$ where $\{F_m\}_{m\in\mathbb{N}}$ is a F{\o}lner sequence with $|F_m|\ge m$. Taking $m\to\infty$, we obtain $ \mathrm{Asc}_{top} (G,\mathcal{U})\le\mathrm{Asc}_\nu ^+(G,\mathcal{U})$ for some $\nu\in\mathcal{M}(X,G)$. 
	
	$\mathbf{Step 2:}$ The general case.
	
	It is well known that there exists a surjective continuous map $\phi_1:C\to X$, where $C$ is a cantor set. Let $C^G$ be the product space equiped with the $G$-shift $G\times C^G\to C^G$. We define 
	$
	Z=\{(z_g)_{g\in G}\in C^G:\phi_1(z_{g_1g_2})=g_1\phi_1(z_{g_2})\text{ }\mathrm{for}\text{ each } g_1,g_2\in G\}
	$	and $\varphi:Z\to X$ with $(z_g)_{g\in G}\in C^G\mapsto\phi_1(z_{e_G})$. It is not hard to obtain that $Z\subseteq C^G$ is a closed invariant subset under the $G$-shift. Additionally, $\varphi:(Z,G)\to(X,G)$ is a factor map between two $G$-systems. According to Step 1, there exists $\nu\in\mathcal{M}(X,G)$ such that $ \mathrm{Asc}_{top} (G,\mathcal{U})\le\mathrm{Asc}_\nu ^+(G,\mathcal{U})$. Let $\mu=\varphi\nu\in\mathcal{M}(X,G)$. Then 
	$$
	\begin{aligned}
		\mathrm{Asc}_\mu ^+(G,\mathcal{U})&=\inf_{\alpha\in\mathcal{P}(X),\alpha\succeq\mathcal{U}}\mathrm{Asc}_\mu (G,\alpha)=\inf_{\alpha\in\mathcal{P}(X),\alpha\succeq\mathcal{U}}\mathrm{Asc}_\nu (G,\varphi^{-1}(\alpha))\\
		&\ge\inf_{\alpha\in\mathcal{P}(X),\alpha\succeq\mathcal{U}}\mathrm{Asc}_\nu ^+(G,\varphi^{-1}(\mathcal{U}))\ge \mathrm{Asc}_{top} (G,\mathcal{U}).
	\end{aligned}
	$$	
\end{proof}

\begin{theorem}\label{thm3.3}
	Let $\pi:(X,T)\to(Y,R)$ be a factor map between two $\mathbb{Z}$-systems and  $\{c^n_S=2^{-n}:S\subseteq n^*\}$ the uniform system of coefficients. Then for every $\mathcal{U}\in \mathcal{C}^o_X $, there exists $\mu\in\mathcal{M}(X,T)$ such that $\mathrm{Asc}_\mathrm{top} (T,\mathcal{U}|Y)= \mathrm{Asc}_\mu^- (T,\mathcal{U}|Y)$.
\end{theorem}
\begin{proof}
	
	First, we assume that $X$ is a zero-dimension space. 
	
	Let $\mathcal{U}=\{U_1,\dots,U_d\}\in \mathcal{C}^o_X $. For any $k\in\mathbb{N}$, $S\subseteq k^*$ and $\mu\in\mathcal{M}(X,T)$, it is easy to show that 
	$$
	\mathrm{Asc}_\mu^+ (T^k,\mathcal{U}_S|Y)=\inf_{s\in\mathbb{N}^{S}}\mathrm{Asc}_\mu \left(T^k,\bigvee_{i\in S}T^{-i}{\alpha_{s}}_{(i)}|Y\right).
	$$
	For any $k\in\mathbb{N}$, $S\subseteq k^*$ and $s\in\mathbb{N}^{S}$, we set 
	$$
	\begin{aligned}
		M(k,s)&:=\left\{\mu:\frac{1}{k}\sum_{0\in S\subseteq k^*}\frac{1}{2^{k-1}}\mathrm{Asc}_\mu \left(T^k,\bigvee_{i\in S}T^{-i}{\alpha_{s}}_{(i)}|Y\right)\ge\mathrm{Asc}_\mathrm{top} (T,\mathcal{U}|Y) \right\},
	\end{aligned}
	$$
	where $\mu\in\mathcal{M}(X,T)$.
	By Theorem \ref{thm3.2}, we shall obtain that there exists $\mu_k\in\mathcal{M}(X,T^k)$ such that $\mathrm{Asc}_{\mu_k}^+ (T^k,\mathcal{U}_S|Y)\ge\mathrm{Asc}_\mathrm{top} (T^k,\mathcal{U}_S|Y)$. Since $\bigvee\limits_{i\in S}T^{-i}{\alpha_{s}}_{(i)}\succeq\mathcal{U}_S$ for each $s\in\mathbb{N}^{S}$, we have 
	$$
	\mathrm{Asc}_{\mu_k}\left(T^k,\bigvee_{i\in S}T^{-i}{\alpha_{s}}_{(i)}|Y\right)\ge\mathrm{Asc}_\mathrm{top} (T^k,\mathcal{U}_S|Y).
	$$
	Let $\nu_{k,S}=\frac{1}{|S|}\sum\limits_{i\in S}\mu_k\circ T^{-i}$. Then $\mu_k\circ T^{-i}\in\mathcal{M}(X,T^k)$ for each $i\in S$. Let $s=({ s}(0),\dots,{ s}(|S|-1))$. We define 
	$$
	P^j_{s}=\underbrace{ s(|S|- j) s(|S|- j+1)\cdots s(|S|- 1)}_{j}\underbrace{ s(0)s(1)\cdot s(|S|-j-1)}_{|S|-j} .
	$$
	Then $P^j_{s}\in\mathbb{N}^{S}$. Set $P^0_{s}=s$. Combining Proposition \ref{prop3.2}, it is not hard to show that 
	$$
	\begin{aligned}
		\mathrm{Asc}_{{\mu_{k}}\cdot T^{-s(j)}} \left(T^k,\bigvee_{i\in S}T^{-i}{\alpha_{s}}_{(i)}|Y\right)&=\mathrm{Asc}_{\mu_{k}} \left(T^k,\bigvee_{i\in S}T^{-i}{\alpha_{P^j_{s}}}_{(i)}|Y\right)\ge\mathrm{Asc}_\mathrm{top} (T^k,\mathcal{U}_S|Y),
	\end{aligned}
	$$
	$$
	\begin{aligned}
		\mathrm{Asc}_{\nu_{k,S}} \left(T^k,\bigvee_{i\in S}T^{-i}{\alpha_{s}}_{(i)}|Y\right)&=\frac{1}{|S|}\sum\limits_{i\in S}\mathrm{Asc}_{\mu_k\circ T^{-i}}\left(T^k,\bigvee_{i\in S}T^{-i}{\alpha_{s}}_{(i)}|Y\right)\ge\mathrm{Asc}_\mathrm{top} (T^k,\mathcal{U}_S|Y).
	\end{aligned}
	$$
	
	This implies that $\nu_{k,S}\in\bigcap\limits_{s\in\mathbb{N}^{S}}M(k,s)$. Write $M(k)=\bigcap\limits_{s\in\mathbb{N}^{S}}M(k,s)\neq\emptyset$ is a subset of $\mathcal{M}(X,T)$.
	
	For each $s\in\mathbb{N}^{S}$, $\bigvee\limits_{i\in S}T^{-i}{\alpha_{s}}_{(i)}$ is a clopen  cover. So the map $\mu\mapsto\mathrm{Asc}_{\mu}\left(T^k,\bigvee\limits_{i\in S}T^{-i}{\alpha_{s}}_{(i)}|Y\right)$ is a u.s.c. function on $\mathcal{M}(X,T^k)$. Since $\mathcal{M}(X,T)\subseteq\mathcal{M}(X,T^k)$, $\mathrm{Asc}_* \left(T^k,\bigvee\limits_{i\in S}T^{-i}{\alpha_{s}}_{(i)}|Y\right)$ is also a u.s.c. function on $\mathcal{M}(X,T)$ by Theorem \ref{thm3.1}. So
	$M(k,s)$ is closed in $\mathcal{M}(X,T)$ for each $s\in\mathbb{N}^{S}$ and $M(k,s)\neq\emptyset$ is closed in $\mathcal{M}(X,T)$.
	
	Now let $k,k_1,k_2\in\mathbb{N}$ with $k_2=kk_1$ and $\mu\in M(k_2)$. Let $B\subseteq k_2^*$. Then for each $i\in B$, there exist $ E\subseteq k^*$, $j\in E$,$b\in\mathbb{N}^B$ and $i_0\in S$ such that $i=k_1j+i_0$ and $ b(i)=s(i_0)$. So
	$$
	\begin{aligned}
		\frac{2}{k_1}&\sum_{0\in S\subseteq{k_1}^*}\frac{1}{2^{k_1}}\mathrm{Asc}_{\mu}\left(T^{k_1},\bigvee_{i\in S}T^{-i}{\alpha_{s}}_{(i)}|Y\right)\\
		&\ge\frac{1}{k_2}\sum_{0\in B\subseteq{k_2}^*}\frac{1}{2^{k_2-1}}\mathrm{Asc}_{\mu} \left(T^{k_2},\bigvee_{i\in B}T^{-i}{\alpha_{b_{(i)}}}|Y\right)\ge\mathrm{Asc}_\mathrm{top}(T,\mathcal{U}|Y).
	\end{aligned}
	$$
	Hence $\mu \in M(k_1,b)$ for each $b\in\mathbb{N}^B$. That is, $\mu\in M(k_1)$. Then $M(k_2)\subseteq M(k_1)$. Since $M(k_1k_2)\subseteq M(k_1)\cap M(k_2)$ for any $k_1,k_2\in \mathbb{N}$, we have $\cap_{k\in\mathbb{N}}M(k)\neq\emptyset$. Take $\mu\in\cap_{k\in\mathbb{N}}M(k)$. For any $s\in\mathbb{N}^S$, we have
	$$\begin{aligned}
	\frac{1}{k}\sum_{0\in S\subseteq k^*}\frac{1}{2^{k-1}}\mathrm{Asc}_\mu^+ \left(T^k,\mathcal{U}_S|Y\right)&=\inf_{s\in\mathbb{N}^S}\frac{1}{k}\sum_{0\in S\subseteq k^*}\frac{1}{2^{k-1}}\mathrm{Asc}_\mu\left(T^k,\bigvee_{i\in S}T^{-i}{\alpha_{s}}_{(i)}|Y\right)\ge\mathrm{Asc}_\mathrm{top}(T,\mathcal{U}|Y).	\end{aligned}$$
	By Proposition \ref{prop3.2}, we have 
	$$
	\mathrm{Asc}_\mu^- \left(T,\mathcal{U}|Y\right)=\lim_{k\to\infty}\frac{1}{k}\sum_{0\in S\subseteq k^*}\frac{1}{2^{k-1}}\mathrm{Asc}_\mu^+ \left(T^k,\mathcal{U}_S|Y\right)\ge\mathrm{Asc}_\mathrm{top}(T,\mathcal{U}|Y).
	$$
	Thus, we have $\mathrm{Asc}_\mu^-\left(T,\mathcal{U}|Y\right)=\mathrm{Asc}_\mathrm{top}(T,\mathcal{U}|Y)$.
	
	Similar to the general case in Theorem \ref{thm3.2}, we finish the proof.

\end{proof}

\subsection{Ergodic decompositions}
This section will demonstrate the ergodic decompositions of two kinds of local measure-theoretic average sample complexity. First, we will present the following lemma based on the idea introduced in \cite{huang2011local}. It plays an important role in the proof of ergodic decompositions.

\begin{lemma}\label{lem3.5}
	Let $\pi:(X,G)\to (Y,G)$ be a factor map between two $G$-system and $\alpha,\beta\in\mathcal{P}(X)$. Assume that $S=\{g_1,\cdots, g_{|S|}\}$ for each $S\subseteq F_n$, then 
	$$\begin{aligned}
\mathrm{Asc} _\mu (G,\alpha \vee \beta )&=\mathrm{Asc} _\mu (G,\beta)+\lim_{n\to\infty}\frac{1}{|F_n|}\sum_{S\subseteq F_n}c^{F_n}_S\sum_{j=2}^{|S|}H_\mu\left(g_j^{-1}\alpha|\bigvee_{i\in S\setminus\{g_j,\cdots,g_{|S|}\}}g_i^{-1}\alpha\vee\beta_S\right).
\end{aligned}$$
\end{lemma}
\begin{proof}
	Fix $S\subseteq F_n$. By the definition, we note that 
	\begin{equation}\label{eq3.4}
		\begin{aligned}
			H_\mu\left(\alpha_S\vee\beta_S\right)=H_\mu(\beta_S)+H_\mu\left(g^{-1}_1\alpha\right|\beta_S)+\sum_{j=2}^{|S|}H_\mu\left(g_j^{-1}\alpha|\bigvee_{i\in S\setminus\{g_j,\cdots,g_{|S|}\}}g_i^{-1}\alpha \vee\beta_S\right).
		\end{aligned}
	\end{equation}
	Since $\lim\limits_{n\to\infty}\frac{1}{|F_n|}\sum\limits_{S\subseteq F_n}c^{F_n}_S H_\mu\left(g^{-1}\alpha\right)=0$ for each $g\in G$, we multiply \eqref{eq3.4} on both side by $c^{F_n}_S$, divide by $|F_n|$ and sum over $S\subseteq F_n$. Then, by letting $n\to\infty$, we can complete the proof.
\end{proof}
%pinsker公式得写一下（经典结论得看一下成不成立）写成G系统
\begin{theorem}\label{thm3.4}
	Let $\mu\in\mathcal{M}(X,G)$ and $\mathcal{U}\in\mathcal{C}_X$. If $\mu =\int_{\Omega }\mu _{\omega }dm(\omega )$ is the ergodic decomposition of $\mu$ then
	$$\mathrm{Asc} _\mu ^+(G,\mathcal{U}|Y )=\int_{\Omega }\mathrm{Asc}_{\mu _{\omega}} ^+(G,\mathcal{U}|Y )dm(\omega ).$$
\end{theorem}
\begin{proof}
	First, we shall show that for each $\alpha\in\mathcal{P}(X)$, we have $\mathrm{Asc} _\mu (G,\alpha|Y )=\int_{\Omega }\mathrm{Asc}_{\mu _{\omega}}(G,\alpha|Y )dm(\omega )$. Let $\{\beta_j\}_{j\in\mathbb{N}}$ be an increasing sequence with $\beta_j\in\mathcal{P}(X)$ and $\lim\limits_{j\to\infty}diam(\beta_j)=0$.  Assume that $S=\{g_1,\cdots, g_{|S|}\}$ for each $S\subseteq F_n$. By Lemma \ref{lem3.5} and the dominated convergence theorem, one has
	$$
	\begin{aligned}
		&\mathrm{Asc} _\mu (G,\alpha|Y )=\lim_{j\to\infty}\lim_{n\to\infty}\frac{1}{|F_n|}\sum_{S\subseteq F_n}c^{F_n}_S\sum_{j=2}^{|S|}H_\mu\left(g_j^{-1}\alpha|\bigvee_{i\in \{g_1,\cdots,g_{j-1}\}}g_i^{-1}\alpha\vee\pi^{-1}(({\beta_j})_{S}) \right)\\
		=&\lim_{j\to\infty}[\mathrm{Asc} _\mu (G,\alpha^-\vee( \pi^{-1}(\beta_j) )-\mathrm{Asc} _\mu (G,\pi^{-1}(\beta_j) )]\\
		=&\lim_{j\to\infty}\int_\Omega[\mathrm{Asc} _{\mu_\omega}(G,\alpha^-\vee( \pi^{-1}(\beta_j) )-\mathrm{Asc} _{\mu_\omega}(G,\pi^{-1}(\beta_j) )]dm(\omega)\\
		=&\int\limits_\Omega\lim_{j\to\infty}\lim_{n\to\infty}\frac{1}{|F_n|}\sum_{S\subseteq F_n}c^{F_n}_S\sum_{j=2}^{|S|}H_{\mu_\omega}\left(g_j^{-1}\alpha|\bigvee_{i\in \{g_1,\cdots,g_{j-1}\}}g_i^{-1}\alpha\vee\pi^{-1}(({\beta_j})_{S}) \right)dm(\omega)\\
		=&\int_\Omega\mathrm{Asc} _{\mu_\omega} (T,\alpha|Y )dm(\omega).
	\end{aligned}
	$$
	Taking an increasing sequence  $\{\gamma_j\}_{j\in\mathbb{N}}\subseteq\mathcal{P}(X)$ such that $\lim\limits_{j\to\infty}diam(\gamma_j)=0$, then
	$$
	\begin{aligned}
		\mathrm{Asc} _\mu (G,X|Y )&=\lim_{j\to\infty}\mathrm{Asc} _\mu (G,\gamma_j|Y )=\lim_{j\to\infty}\int_\Omega\mathrm{Asc} _{\mu _\omega}(G,\gamma_j|Y )dm(y)\\
		&=\int_\Omega\lim_{j\to\infty}\mathrm{Asc} _{\mu _\omega}(G,\gamma_j|Y )dm(y)=\int_\Omega\mathrm{Asc} _{\mu _\omega}(G,X|Y )dm(y).
	\end{aligned}
	$$
	
	Now we shall prove that for each $\mathcal{U}\in\mathcal{C}_X$, we have $\mathrm{Asc} _\mu ^+(G,\mathcal{U}|Y )=\int_{\Omega }\mathrm{Asc}_{\mu _{\omega}} ^+(G,\mathcal{U}|Y )dm(\omega ).$ 
    
    Since $X$ is a compact metric space, there exist $\{\alpha_n\}_{n\in\mathbb{N}}\subseteq\mathcal{U}^*$ which is $L^1(X,\mathcal{B}(X),\nu)$-dense in $\mathcal{U}^*$ for each $\nu\in\mathcal{M}(X,G)$. So for each $\nu\in\mathcal{M}(X,G)$, 
	$
	\mathrm{Asc} _\nu ^+(G,\mathcal{U}|Y )=\inf_{n\in\mathbb{N}}\mathrm{Asc} _\nu ^+(G,\alpha_n|Y ).
	$
	Set $\alpha_n=\{A^n_1,\dots,A^n_d\}$, $n\in\mathbb{N}$. By Fatou's lemma, we have
	$$
	\begin{aligned}
		\mathrm{Asc} _\nu ^+(G,\mathcal{U}|Y )&=\inf_{n\in\mathbb{N}}\mathrm{Asc} _\nu ^+(G,\alpha_n|Y )=\inf_{n\in\mathbb{N}}\int_\Omega\mathrm{Asc} _{\nu_\omega} ^+(G,\alpha_n|Y )dm(\omega)\\
		&\ge\int_\Omega\inf_{n\in\mathbb{N}}\mathrm{Asc} _{\nu_\omega} ^+(G,\alpha_n|Y )dm(\omega)=\int_\Omega\mathrm{Asc} _{\nu_\omega} ^+(G,\mathcal{U}|Y )dm(\omega).
	\end{aligned}
	$$
	
	On the other hand, for each $\varepsilon>0$, $n\in\mathbb{N}$ and $S\subseteq F_n$, we define $$B_n^\varepsilon=\{\omega\in\Omega:\mathrm{Asc} _{\nu_\omega} ^+(G,\alpha_n|Y )\le\mathrm{Asc} _{\nu_\omega} ^+(G,\mathcal{U}|Y )+\varepsilon\}.$$ We note that $m(\Omega\Delta \bigcup_{n\in\mathbb{N}}B_n^\varepsilon)=0$. There exist $\{\Omega_j\}_{j\in\mathbb{N}}\in\mathcal{P}(\Omega)$ with $m(\Omega)>0$ and $\{\alpha_{n_j}\}_{j\in \mathbb{N}}\subseteq \{\alpha_{n}\}_{n\in \mathbb{N}}$ such that for $m$-a.e. $\omega\in\Omega_j$ and $j\in\mathbb{N}$, we have $\mathrm{Asc} _{\nu_\omega} ^+(G,\alpha_{n_j}|Y )\le\mathrm{Asc} _{\nu_\omega} ^+(G,\mathcal{U}|Y )+\varepsilon$. So for each $j\in\mathbb{N}$, we define $\nu_j\in\mathcal{M}(X,T)$ by $\nu_j=\frac{1}{m(\Omega_j)}\int_{\Omega_j}\nu_\omega dm(\omega)$. So 
	$$
	\begin{aligned}
		\mathrm{Asc} _{\nu} ^+(G,\alpha_{n_j}|Y )&=\frac{1}{m(\Omega_j)}\int_{\Omega_j}\mathrm{Asc} _{\nu_\omega} ^+(G,\alpha_{n_j}|Y ) dm(\omega)\le\frac{1}{m(\Omega_j)}\int_{\Omega_j}\mathrm{Asc} _{\nu_\omega} ^+(G,\mathcal{U}|Y ) dm(\omega)+\varepsilon.
	\end{aligned}
	$$
	Since $\{\nu_n\}$ is mutually singular, there exist $\{X_n\}_{n\in\mathbb{N}}$ such that $\nu_n(X_n)=1$ and $\nu_n(X_k)=0$ for $k\neq n$. For $1\le i\le d$, we define $A_i=\bigcup_{j\ge1}(X_j\cap A_i^{n_j})$, then $\alpha=\{A_1,\dots,A_d\}\in\mathcal{U}^*$ and 
	$$
	\begin{aligned}
		\mathrm{Asc} _{\nu} ^+(G,\mathcal{U}|Y )&\le\mathrm{Asc} _{\nu} ^+(G,\alpha|Y )=\sum_{j\in\mathbb{N}}m(\Omega_j)\mathrm{Asc} _{\nu_j} ^+(T,\alpha_{n_j}|Y )\le\int_{\Omega_j}\mathrm{Asc} _{\nu_\omega} ^+(T,\mathcal{U}|Y )dm(\omega)+\varepsilon .
	\end{aligned}
	$$
	Letting $\varepsilon\to0$, we obtain that $\mathrm{Asc} _{\nu} ^+(G,\mathcal{U}|Y )\le\int_\Omega\mathrm{Asc} _{\nu_\omega} ^+(G,\mathcal{U}|Y ) dm(\omega)$. 
\end{proof}

\begin{theorem}
	
	Let $\mu\in\mathcal{M}(X,T)$, $\mathcal{U}\in\mathcal{C}_X$ and $\{c^n_S=\frac{1}{2^n}:S\subseteq n^*\}$. If $\mu =\int_{\Omega }\mu _{\omega }dm(\omega )$ is the ergodic decomposition of $\mu$, then
	$$\mathrm{Asc} _\mu ^-(T,\mathcal{U}|Y )=\int_{\Omega }\mathrm{Asc}_{\mu _{\omega}} ^-(T,\mathcal{U}|Y )dm(\omega ).$$
\end{theorem}
\begin{proof}
	By Proposition \ref{prop3.2} and Theorem \ref{thm3.4}, we have
	$$
	\begin{aligned}
		\mathrm{Asc} _\mu ^-(T,\mathcal{U}|Y )&=\lim_{n\to\infty}\frac{1}{n}\sum_{0\in S\subseteq n^*}\frac{1}{2^{n-1}}\mathrm{Asc} _{\mu}^+(T^n,\mathcal{U}_S|Y )=\lim_{n\to\infty}\sum_{0\in S\subseteq n^*}\frac{1}{2^{n-1}}\int_{\Omega}\mathrm{Asc} _{\mu_\omega}^+(T^n,\mathcal{U}_S|Y ) dm(\omega)\\
		&=\int_{\Omega}\lim_{n\to\infty}\frac{1}{n}\sum_{0\in S\subseteq n^*}\frac{1}{2^{n-1}}\mathrm{Asc} _{\mu_\omega}^+(T^n,\mathcal{U}_S|Y ) dm(\omega)=\int_{\Omega}\mathrm{Asc} _{\mu_\omega}^-(T,\mathcal{U}|Y ) dm(\omega).
	\end{aligned}
	$$\end{proof}

\section{Equivalence of two kinds of measure-theoretic average sample complexity }\label{sec4}
\subsection{The case of $\mathbb{Z}$-systems}

Let $\pi:(X,T)\to(Y,R)$ be a factor map between two $\mathbb{Z}$-systems. In this subsection, we shall prove that $\mathrm{Asc} _\mu ^+(T,\mathcal{U}|Y)= \mathrm{Asc} _\mu ^-(T,\mathcal{U}|Y)$ for each $\mu\in\mathcal{M}(X,T)$ and $\mathcal{U}\in\mathcal{C}_X$ when $\{c^n_S=2^{-n}:S\subseteq n^*\}$. 

\begin{lemma}\label{lem4.1}\cite[Rohlin's lemma]{glasner2003ergodic}
	Let $(X,T)$ be a $\mathbb{Z}$-systems and $\mu\in\mathcal{M}^e(X,T)$. Suppose $\mu$ is a non-atomic. Then for each $N\in\mathbb{N}$ and $\varepsilon>0$, there exists a Borel subset $D$ of $X$ such that $D,TD,\dots,T^{N-1}D$ are pairwise disjoint and $\mu \left ( \bigcup\limits_{i=0}^{N-1}T^iD  \right ) >1-\varepsilon $.
\end{lemma}
By Lemma \ref{lem4.1} and Theorem \ref{thm3.3}, we shall present the following theorem which is the variational principle of the conditional local average sample complexity for $\mathbb{Z}$-systems. 

\begin{theorem}\label{thm4.1}
	Let $\pi:(X,T)\to(Y,R)$ be a factor map between two $\mathbb{Z}$-systems, $\mathcal{U}\in\mathcal{C}_X$. Then $\mathrm{Asc}_\mathrm{top}(T,\mathcal{U}|Y)\ge \mathrm{Asc}_\mu ^+(T,\mathcal{U}|Y)$. When $\mathcal{U}\in\mathcal{C}_X^o$, we have
	$$
	\mathrm{Asc}_\mathrm{top}(T,\mathcal{U}|Y)=\max_{\mu\in\mathcal{M}(X,T)}\mathrm{Asc}_\mu ^+(T,\mathcal{U}|Y).
	$$
\end{theorem}
\begin{proof}
	We are inspired by the approach presented in \cite{glasner2003ergodic,huang2004entropy,romagnoli2018local}. By Theorem \ref{thm3.4}, we need to proved that $\mathrm{Asc}_\mathrm{top}(T,\mathcal{U}|Y)\ge \mathrm{Asc}_\mu ^+(T,\mathcal{U}|Y)$ for each non-atomic measure $\mu\in\mathcal{M}^e(X,T)$.
	
	Let $\mu\in\mathcal{M}^e(X,T)$ be a non-atomic measure, $\nu=\pi\mu$ and $\mathcal{U}=\{U_1,\dots,U_d\}\in\mathcal{C}_X$. Then $\nu\in\mathcal{M}^e(Y,R)$. Let $\mu=\int_Y\mu_yd\nu(y)$ be the disintegration of $\mu$ with respect to $\nu$. It is not hard to show that we can take $\beta \succeq\mathcal{U}$ such that $|\beta|_{\mu_y}\le N(\mathcal{U}|Y)$ for $\nu$-a.e. where $|\beta|_{\mu_y}=|\{B\in\beta:\mu_y(B)>0\}|$ for each $y\in Y$. Then for each $S\subseteq N^*$, we have $|\beta_S|_{\mu_y}\le N(\mathcal{U}_S|Y)$ for $\nu$-a.e. $y\in Y$. By the definition of $\mathrm{Asc}_\mathrm{top}(T,\mathcal{U}|Y)$, for $\varepsilon>0$, we choose $N\in\mathbb{N}$ large enough such that 
	$$
	\sum_{S\subseteq N^*}c^N_S\log N(\mathcal{U}_S|Y)\le N\left(\mathrm{Asc}_\mathrm{top}(T,\mathcal{U}|Y)+\varepsilon\right).
	$$
	Let $0<\delta<1$ small enough such that $2\sqrt{\delta}\log k<\varepsilon $. By Lemma \ref{lem4.1}, there exists $D\subseteq \mathcal{B}(X)$ with $D, TD,\dots,T^{N-1}D$ pairwise disjoint such that $\mu\left(\bigcup\limits_{i=0}^{N-1}T^{-i}D\right)>1-\varepsilon$. Now we define the partition of $D$, $\beta_D=\{B\cap D:B\in\beta\}$. We use the partition $\beta_D$ to define a partition of $X$ by $\alpha=\{A_1,\dots,A_d\}$, which assigning the set $A_i$ all the sets of the form with
	$
	P=P_{i}\subseteq{U_i}\cap D,
	$
	where $P\in\beta_D$. On the remainder of the space $\alpha$ can be taken to be any partition refining $\mathcal{U}$. we note that $\alpha_S\cap D={(\beta_D)}_S$. Set ${(\beta_D)}_S=\beta_{S,D}$, then
	$$
	|\alpha_S\cap D|_{\mu_y}=|\beta_{S,D}|_{\mu_y}\le|\beta_{S}|_{\mu_y}\le N(\mathcal{U}_S|Y)\text{ for }\nu\text{-a.e.}\text{ } y\in Y.
	$$
	Fix $\mu\in\mathcal{M}^e(X,T)$, we shall show that $\mathrm{Asc}_\mu ^+(T,\mathcal{U}|Y)\le\mathrm{Asc}_\mathrm{top}(T,\mathcal{U}|Y)+3\varepsilon$. Let $E=\bigcup\limits_{i=0}^{N-1}T^{i}D$. Then $\mu(E)>1-\delta$. Let $n\gg N$ such that $\lim_{n\to \infty}\frac{N}{n}=0$ and $K\subseteq n^*$, then we set
	$
	G_{n,K}=\left\{x\in X:\frac{1}{|K|}\sum_{i\in K}1_E(T^ix)>1-\sqrt{\delta} \right\}.
	$ So
	$$
	\begin{aligned}
		\mu(G_{n,K})+(1-\sqrt{\delta})(1-\mu(G_{n,K}))&\ge\int\limits_{X}\frac{1}{|K|}\sum_{i\in K}1_E(T^ix)d\mu(x)=\mu(E)>1-\delta.
	\end{aligned}
	$$
	So we have $\mu(G_{n,K})>1-\sqrt\delta$. For each $x\in G_{n,K}$, we define $S_{n,K}(x)=\{i\in K:T^ix\in D\}$. $|\cup_{j\in S}(S_{n,K}(x)+j)|=\sum\limits_{i\in S}1_E(T^ix)>|K|(1-\sqrt\delta)$. By definition, fix $S\subseteq N^*$, $i\in\{i\in K:T^ix\notin \bigcup_{j\in S}T^{-j}D\}$ if and only if $i\notin S_{n,K}(x)+j$ for any $j\in S$. Thus,
	$
	\left |  K\setminus \bigcup_{j\in S}(S_{n,K}(x)+j)\right | \le|K|\sqrt\delta+|S|.
	$
	
	Let $\mathcal{F}_{n,K}=\{S_{n,K}(x):x\in G_{n,K}\}$ and $F=\{s_1,\dots,s_l\}\in\mathcal{F}_{n,K}$. It is not hard to show that $F\cap(F+i)=\emptyset$, $i\in\{0,\dots N-1\}$. Since $D, TD,\dots,T^{N-1}D$ are pairwise disjoint, we have $|F|\le\frac{|K|}{|S|}+1$. Set $a_{n,K}=[\frac{|K|}{|S|}]+1$, then
	$
	|\mathcal{F}_{n,K}|\le \sum_{j=1}^{a_{n,K}}\binom{|K|}{j} \le K\binom{|K|}{a_{n,K}}.
	$
	By Stirling's formula, we have 
	$\lim_{n\to\infty}\frac{1}{n}\log\left(|K|\binom{|K|}{a_{n,K}}\right)\le \varepsilon.$
	So we have 
	$$
	\limsup_{n\to\infty}\frac{1}{n}\log(|\mathcal{F}_{n,K}|+1)\le\lim_{n\to\infty}\frac{1}{n}\log\left(|K|\binom{|K|}{a_{n,K}}\right)\le\varepsilon.
	$$
	We define $B_F=\{x\in G_{n,K}:S_{n,K}(x)=F,F\in\mathcal{F}_{n,K}\}$ and $\gamma=\{B_F:F\in\mathcal{F}_{n,K}\}$. It is easy to show that $\gamma$ is the partition of $ G_{n,K}$. For $F\in\mathcal{F}_{n,K}$, let $F=\{f_1,\cdots,f_l\}$ and $H_{FS}=K\setminus\bigcup\limits_{j\in S}(F+j)$ with $|H_{FS}|\le|K|\sqrt\delta+|S|$. Since $B_F\subseteq  G_{n,K}\cap \bigcap\limits_{j\in S}
	T^{-j}D$, then
	$$
	\begin{aligned}
		|\alpha_K\cap B_F|_{\mu_y}&\le\prod _{j=1}^l\left | \alpha_S\cap D \right |_{T^{j}\mu_y} \prod_{r\in H_{FS}}|\alpha|_{\mu_y}\le N(\mathcal{U}_S|Y)^{l}d^{| H_{FS}|}\le d^{|K|\sqrt\delta+|S|}(N(\mathcal{U}_S|Y))^{|K|/|S|+1}.
	\end{aligned}
	$$
	Set $b_{n,K,S}=k^{|K|\sqrt\delta+|S|}(N(\mathcal{U}_S|Y))^{|K|/|S|+1}$, then we obtain that 
	\begin{equation}\label{111}
		\begin{aligned}
			&H_{\mu}(\alpha_K\cap B_F|Y)=\int\limits_YH_{\mu_y}(\alpha_K\cap B_F|Y)d\nu(y)\le\int\limits_Y\mu_y(B_F)(\log|\alpha_K\cap B_F|_{\mu_y}-\log\mu_y(B_F))d\nu(y)\\
			\le&\int\limits_Y\mu_y(B_F)(\log b_{n,K,S}-\log\mu_y(B_F))d\nu(y)=\mu(B_F)(\log b_{n,K,S})-\int\limits_Y\mu_y(B_F)\log(\mu_y(B_F))d\nu(y).
		\end{aligned}
	\end{equation}
	Moreover, we have 
	\begin{equation}\label{112}
		\begin{aligned}
			H_{\mu}(\alpha_K\cap( X\setminus G_{n,K})|Y)&\le\mu(X\setminus G_{n,K})(\log b_{n,K,S})-\int\limits_Y\mu_y(X\setminus G_{n,K})\log(\mu_y(X\setminus G_{n,K}))d\nu(y).
		\end{aligned}
	\end{equation}
	Let $\Delta=\{B_F\}_{F\in\mathcal{F}_{n,K}}\cup\{X\setminus G_{n,K}\}$. Combining (\ref{111}) and (\ref{112}), we obtain that %第二行改一些6.10
	$$
	\begin{aligned}
		H_{\mu}(\alpha_K|Y)
		&\le \log b_{n,K,S}+\sqrt{\delta}|K|\log k+\int\limits_y|\log\Delta|d\nu(y)=\log b_{n,K,S}+\sqrt{\delta}|K|\log k+\log (|\mathcal{F}_{n,K}|+1).
	\end{aligned}
	$$
	For each $S\subseteq N^*$ and $S\neq\emptyset $, %假设n远大于N使得 N/n极限趋于0
	this implies that
	$$
	\begin{aligned}
		\mathrm{Asc}_\mu ^+(T,\mathcal{U}|Y)&
		\le\limsup_{n\to\infty}\frac{1}{n}\sum_{K\subseteq n^*}c^n_K[\log(k^{|K|\sqrt\delta+|S|}N(\mathcal{U}_S|Y)^{|K|/ |S|+1})+\sqrt{\delta}|K|\log k+\log (|\mathcal{F}_{n,K}|+1]\\
		&\le\sum_{S\subseteq N^*}c^N_S\log N(\mathcal{U}_S|Y)+3\varepsilon.
	\end{aligned}
	$$
	Since $\varepsilon$ is arbitrary, letting $N\to\infty$,  $\mathrm{Asc}_\mu ^+(T,\mathcal{U}|Y)\le\mathrm{Asc}_\mathrm{top}(T,\mathcal{U}|Y)$.
\end{proof}

The following theorem shows that two kinds of measure-theoretic average sample complexity of $\mathbb{Z}$-systems are equal when $\{c^n_S=\frac{1}{2^n}: S\subseteq n^*\}$.
\begin{theorem}\label{thm4.2}
	Let $\pi:(X,T)\to(Y,R)$ be a factor map between two $\mathbb{Z}$-systems and $\{c^n_S=\frac{1}{2^n}:S\subseteq n^*\}$. If $(X,T)$ is uniquely ergodic, then the following holds:
	\item[$\mathrm{(1)}$] $\mathrm{Asc}_\mathrm{top}(T,\mathcal{U}|Y)=\mathrm{Asc}_\mu ^+(T,\mathcal{U}|Y)=\mathrm{Asc}_\mu ^-(T,\mathcal{U}|Y)$ for each $\mathcal{U}\in\mathcal{C}^o_X$.
	\item[$\mathrm{(2)}$] $\mathrm{Asc}_\mu ^+(T,\mathcal{U}|Y)=\mathrm{Asc}_\mu ^-(T,\mathcal{U}|Y)$ for each $\mathcal{U}\in\mathcal{C}_X$.
\end{theorem}
\begin{proof} (1) Combining Theorem \ref{thm3.3} and Theorem \ref{thm4.1}, we obtain the first result.
	
	(2)Fix $\mathcal{V}\in\mathcal{C}_X$ and $\varepsilon>0$. Put $|\mathcal{V}|=R$. By Proposition \ref{equation}, there exists $\delta_1>0$ such that for every $\mathcal{U}_1\in\mathcal{C}_X$ with $|\mathcal{U}_1|=R$, when $\mu(\mathcal{U}_1\bigtriangleup\mathcal{V})<\delta_1$, we have 
	\begin{equation}\label{eq6.1}
		\begin{aligned}
			\mathrm{Asc}_\mu ^+(T,\mathcal{V}|Y)\le  \mathrm{Asc}_\mu ^+(T,\mathcal{U}_1|Y)+\varepsilon.
		\end{aligned}
	\end{equation}
	
	By Proposition \ref{prop3.2}, there exists $M\in\mathbb{N}$ such that 
	\begin{equation}
		\begin{aligned}
			\frac{1}{M}\sum_{0\in S\subseteq M^*} \frac{1}{2^{M-1}}\mathrm{Asc}_\mu ^+(T^M,\mathcal{V}_S|Y)\le  \mathrm{Asc}_\mu ^-(T,\mathcal{V}|Y)+\frac{\varepsilon }{2} 
		\end{aligned}.
	\end{equation}
	
	By Proposition \ref{equation}, there exists $\delta_{2,S}>0$ such that for every $\mathcal{U}_{2,S}\in\mathcal{C}_X$ with $|\mathcal{U}_{2,S}|=|\mathcal{V}_S|$, when $\mu(\mathcal{U}_{2,S}\bigtriangleup\mathcal{V}_S)<\delta_{2,S}$, we have 
	\begin{equation}
		\begin{aligned}
			\frac{1}{M}\sum_{0\in S\subseteq M^*} \frac{1}{2^{M-1}} \mathrm{Asc}_\mu ^+(T^M,\mathcal{U}_{2,S}|Y)\le \frac{1}{M}\sum_{0\in S\subseteq M^*}  \frac{1}{2^{M-1}}\mathrm{Asc}_\mu ^+(T^M,\mathcal{V}_S|Y)+\frac{\varepsilon }{2} .
		\end{aligned}
	\end{equation}
	
	Now we consider $\mathcal{U}\in\mathcal{C}_X^o$ with $|\mathcal{U}|=R$ such that $\mu(\mathcal{U}\bigtriangleup\mathcal{V})<\min\{\delta_1,\frac{\delta_{2,S}}{R^{|S|}},S\in M^*\}$, by Equation (\ref{eq6.1}), one has 
	\begin{equation}\label{eq6.4}
		\begin{aligned}
			\mathrm{Asc}_\mu ^+(T,\mathcal{V}|Y)\le  \mathrm{Asc}_\mu ^+(T,\mathcal{U}Y)+\varepsilon.
		\end{aligned}
	\end{equation}
	Since $\mu(\mathcal{U}_S\bigtriangleup\mathcal{V}_S)\le |\mathcal{V}_S|\mu(\mathcal{U}\bigtriangleup\mathcal{V})<\delta_{2,S}$, so
	\begin{equation}\label{eq6.5}
		\begin{aligned}
			\frac{1}{M}\sum_{0\in S\subseteq M^*}  \frac{1}{2^{M-1}}\mathrm{Asc}_\mu ^+(T^M,\mathcal{U}_{S}|Y)&\le \frac{1}{M}\sum_{0\in S\subseteq M^*}  \frac{1}{2^{M-1}}\mathrm{Asc}_\mu ^+(T^M,\mathcal{V}_S|Y)+\frac{\varepsilon }{2}\le\mathrm{Asc}_\mu ^-(T,\mathcal{V}|Y)+ \varepsilon.
		\end{aligned}
	\end{equation}
	Thus, combining above equations, Theorem \ref{thm4.2} (1), we have
	$$
	\begin{aligned}
		\mathrm{Asc}_\mu ^+(T,\mathcal{V}|Y)&\le\mathrm{Asc}_\mu ^+(T,\mathcal{U}|Y)+\varepsilon=\mathrm{Asc}_\mu ^-(T,\mathcal{U}|Y)+\varepsilon\le\frac{1}{M}\sum_{0\in S\subseteq M^*}  \frac{1}{2^{M-1}}\mathrm{Asc}_\mu ^-(T^M,\mathcal{U}_S|Y)+\varepsilon\\
		&\le \frac{1}{M}\sum_{0\in S\subseteq M^*} \frac{1}{2^{M-1}} \mathrm{Asc}_\mu ^+(T^M,\mathcal{U}_S|Y)+\varepsilon\le\mathrm{Asc}_\mu ^-(T,\mathcal{V}|Y)+2 \varepsilon.
	\end{aligned}
	$$
	
	Since $\varepsilon$ is arbitrary, we have $\mathrm{Asc}_\mu ^+(T,\mathcal{V}|Y)=\mathrm{Asc}_\mu ^-(T,\mathcal{V}|Y)$. We complete the proof.
\end{proof}

The following theorem shall expand unique ergodic $\mathbb{Z}$-systems to the general case.
\begin{theorem}\label{thm4.3}
	Let $\pi:(X,T)\to(Y,R)$ be a factor map between $\mathbb{Z}$-systems, $\mu\in\mathcal{M}(X,T)$, $\{c^n_S=\frac{1}{2^n}:S\subseteq n^*\}$ and $\mathcal{U}\in\mathcal{C}_X$. Then
	$$\mathrm{Asc}_\mu ^+(T,\mathcal{U}|Y)=\mathrm{Asc}_\mu ^-(T,\mathcal{U}|Y).$$
\end{theorem}

\begin{proof}
	Now we consider an ergodic measure $\mu$. Let $\nu=\pi\circ\mu$. According to the relative Jewett-Krieger theorem \cite{weiss1985strictly}, there exists a unique ergodic model $\widehat{\pi}:(\widehat{X},\widehat{T},\widehat{\mu})\to(\widehat{Y},\widehat{S},\widehat{\nu})$ of $\pi:(X,T,\mu)\to(Y,S,\nu)$, i.e., there exist measure-theoretical isomorphisms $\phi :(X,T,\mu)\to(\widehat{X},\widehat{T},\widehat{\mu})$ and $\psi :(Y,S,\nu)\to(\widehat{Y},\widehat{S},\widehat{\nu})$ such that $\phi\widehat{\pi}=\pi\psi$. Then
	
	\begin{equation}
		\begin{aligned}
			\mathrm{Asc}_\mu ^\pm (T,\mathcal{U}|Y)=\mathrm{Asc}_{\widehat{\mu} }^{\pm} (\widehat{T},\phi(\mathcal{U})|\widehat{Y}).
		\end{aligned}
	\end{equation}
	By Theorem \ref{thm4.2}, we have
	\begin{equation}
		\begin{aligned}
			\mathrm{Asc}_{\widehat{\mu} }^- (\widehat{T},\phi(\mathcal{U})|\widehat{Y})=\mathrm{Asc}_{\widehat{\mu} }^+(\widehat{T},\phi(\mathcal{U})|\widehat{Y}).
		\end{aligned}
	\end{equation}
	Combining the above equations, we get 
	$\mathrm{Asc}_\mu ^+(T,\mathcal{U}|Y)=\mathrm{Asc}_\mu ^-(T,\mathcal{U}|Y).$
\end{proof}

\subsection{Equivalence of two kinds of measure-theoretic average sample complexity of amenable group}
In this subsection, we shall introduce an orbital approach to local average sample complexity for amenable group actions. We will show that two kinds of measure-theoretic average sample complexity of covers are equal for an amenable group $G$ by combining the orbital approach with the equivalence of measure-theoretic average sample complexity in the case of $G=\mathbb{Z}$.

Let $\varepsilon>0$ and $\mathcal{T}$, $\mathcal{S}\in I(\mathcal{R})$. We write $\mathcal{T}\subseteq_{\varepsilon}\mathcal{S}$ if there is $A\in \mathcal{B}_X$ such that $\mu(A)>1-\varepsilon$ and 
$$|\{y\in\mathcal{S}(x):\mathcal{T}(y)\subseteq \mathcal{S}(x)\}|>(1-\varepsilon)|\mathcal{S}(x)|\text{ for each } x\in A.$$
The following lemmas are proved in \cite{danilenko2001entropy}.

\begin{lemma}\label{lem4.2}\cite[Lemma 4.1] {huang2011local}
	Let $\varepsilon>0$ and $\mathcal{T}$, $\mathcal{S}\in I(\mathcal{R})$. Then $A_0=\{x\in A:\mathcal{T}(x)\subseteq\mathcal{S}(x)\}$ is $\mathcal{T}$-invariant, $\mu(A_0)>1-2\varepsilon$ and $|\mathcal{S}(x)\cap A_0|>(1-\varepsilon)|\mathcal{S}(x)|$ for each $x\in A_0$.
\end{lemma}

\begin{lemma}\label{lem4.3}\cite[Lemma 4.2] {huang2011local}
	Let $\varepsilon >0$ and $\mathcal{R} $ be hyperfinite with $\{\mathcal{R}_n\} _{n\in \mathbb{N} }$ a filtration of $\mathcal{R} $, then the following holds:
	\item[$\mathrm{(i)}$] When $\Gamma \subseteq [\mathcal{R}] $ is a countable subset with $|(\Gamma x)|<+\infty $ for $\mu$-a.e. $x\in X$, then there exists an $\mathcal{R}_n$-invariant subset $A_n$ such that $\mu(A_n)>1-\varepsilon$ and 
	$$\left | \left \{ y\in \mathcal{R}_n(x):\Gamma y\subseteq\mathcal{R}_n(x)  \right \}  \right | >(1-\varepsilon )|\mathcal{R}_n(x)|\text{ for each } x\in A_n \text{, as } n\to\infty.$$ 
	\item[$\mathrm{(ii)}$] If $\mathcal{S}\in I(\mathcal{R})$ then $\mathcal{S}\subseteq_\varepsilon\mathcal{R}_n$, if $n$ is large enough.
\end{lemma}

The following definitions introduce the local average sample complexity with orbital approach.
\begin{definition}
	Let $(Y,\mathcal{B} _Y,\nu )$ be Lebesgue space and $\phi :\mathcal{R} \to Aut (Y,\nu )$ a cocycle. For $
	\mathcal{U} \in \mathcal{C} _{X\times Y}$, we define 
	$$\mathrm{Asc} _\nu ^-(\mathcal{S} ,\phi ,\mathcal{U})=\int_X\frac{1}{|\mathcal{S} (x)|} \sum _{B\subseteq\mathcal{S} (x)}c^{\mathcal{S} (x)}_BH_\nu \left(\bigvee_{y\in B}\phi (x,y)\mathcal{U}_y\right)d\mu(x)\text{, }$$
	
	$$\mathrm{Asc} _\nu^+ (\mathcal{S} ,\phi ,\mathcal{U})=\inf_{\alpha\in \mathcal{P} _{X\times Y},\alpha \succeq \mathcal{U}}\mathrm{Asc} _\nu^- (\mathcal{S} ,\phi ,\alpha), $$
	where $\{\mathcal{U}_y\}_{y\in X}\subseteq\mathcal{C}_Y$ with $\{y\}\times \mathcal{U}_y=\mathcal{U}\cap(\{y\}\times Y)$.
	
	Then we define the $\nu^-$-average sample complexity $\mathrm{Asc} _\nu^-(\phi ,\mathcal{U})$ and $\nu^+$-average sample complexity $\mathrm{Asc} _\nu^+(\phi ,\mathcal{U})$, respectively, by
	$$\mathrm{Asc} _\nu^-(\phi ,\mathcal{U})=\inf _{\mathcal{S}\in I(\mathcal{R})} \mathrm{Asc} _\nu^-(\mathcal{S} ,\phi ,\mathcal{U})\text{, }\mathrm{Asc} _\nu^+(\phi ,\mathcal{U})=\inf _{\mathcal{S}\in I(\mathcal{R})} \mathrm{Asc} _\nu^+(\mathcal{S} ,\phi ,\mathcal{U}).$$
	
\end{definition}
\begin{remark}Let $\beta \in\mathcal{P} _{X\times Y}$ and $\mathcal{U},\mathcal{V}\in \mathcal{C}  _{X\times Y}$.
	\item[$\mathrm{(1)}$] It is easy to obtain that $\mathrm{Asc} _\nu^-(\mathcal{S},\phi ,\beta)=\mathrm{Asc} _\nu^+(\mathcal{S},\phi ,\beta)\text{ and }\mathrm{Asc} _\nu^-(\phi ,\beta)=\mathrm{Asc} _\nu^+(\phi ,\beta).$
	\item[$\mathrm{(2)}$] If $\mathcal{U} \succeq \mathcal{V} $ then $\mathrm{Asc} _\nu^-(\mathcal{S},\phi ,\mathcal{U})\ge\mathrm{Asc} _\nu^-(\mathcal{S},\phi ,\mathcal{V})$ and $\mathrm{Asc} _\nu^+(\mathcal{S},\phi ,\mathcal{U})\ge\mathrm{Asc} _\nu^+(\mathcal{S},\phi ,\mathcal{V})$.
	\item[$\mathrm{(3)}$] It is not hard to show that $\mathrm{Asc} _\nu^+(\mathcal{S},\phi ,\mathcal{U})\ge\mathrm{Asc} _\nu^-(\mathcal{S},\phi ,\mathcal{U})$.
\end{remark}
\begin{proposition}\label{prop4.1}
	Let $(Y,\mathcal{B} _Y,\nu )$ be a Lebesgue space, $\beta :\mathcal{S} \to Aut (Z,\kappa )$ a cocycle, $\mathcal{S}\in I(\mathcal{R})$, $\sigma :Z\times X\to X\times Z,(x,z)\mapsto (z,x)$ the flip and $\mathcal{U}\in\mathcal{C}_{X\times Y}$.
	\item[$\mathrm{(i)}$] Let $\alpha':\sigma ^{-1}\mathcal{S} (\beta)\sigma \to Aut(Y,\nu)$ and $\alpha :\mathcal{S}\to Aut(Y,\nu ) $ be two cocycles with $\alpha '((z,x),(z',x'))=\alpha (x,x')$ when $((z,x),(z',x'))\in\sigma ^{-1}\mathcal{S} (\beta )\sigma$. Then $\mathrm{Asc} _\nu^-(\sigma ^{-1}\mathcal{S} (\beta )\sigma ,\alpha' ,Z\times\mathcal{U})=\mathrm{Asc} _\nu^-(\mathcal{S},\alpha ,\mathcal{U}).$
	\item[$\mathrm{(ii)}$] Let $\alpha'':\mathcal{S} (\beta )\to Aut(Y,\nu )$ and $\alpha :\mathcal{S}\to Aut(Y,\nu )$ be two cocycles with $\alpha ''((z,x),(z'',x''))=\alpha (x,x'')$ when $((z,x),(z',x''))\in\mathcal{S} (\beta )$. Then if $\mathcal{U}''\in\mathcal{C}_{X\times Z\times Y}$ with $\mathcal{U}''_{(x,z)}=\mathcal{U}_x$ for each $(x,z)\in X\times Z$, then $\mathrm{Asc} _\nu^-(\mathcal{S} (\beta ) ,\alpha'' ,\mathcal{U}'')=\mathrm{Asc} _\nu^-(\mathcal{S},\alpha ,\mathcal{U})$.
\end{proposition}
\begin{proof}
	Since the proof for (ii) is similar to (i), we prove (i). Let $\mathcal{U}\in\mathcal{C}_{X\times Y}$. Then
	$$
	\begin{aligned}
		&\mathrm{Asc} _\nu^-(\sigma ^{-1}\mathcal{S} (\beta )\sigma ,\alpha' ,Z\times \mathcal{U})=\int\limits_{Z\times X}\frac{1}{|\mathcal{S} (x)| }\sum \limits_{B\subseteq\mathcal{S} (\beta )(x,z)}c^{\mathcal{S} (\beta )(x,z)}_{B}H_\nu\left ( \bigvee \limits_{(x',z')\in B} \alpha(x,x')\mathcal{U}_{x'} \right )d\kappa \times \mu(z,x)\\
		=&\int\limits_{ X}\frac{1}{|\mathcal{S} (x)| }\sum \limits_{B\subseteq\mathcal{S} (x)}c^{\mathcal{S} (x)}_{B}H_\nu\left ( \bigvee _{x'\in B} \alpha(x,x')\mathcal{U}_{x'} \right )d\mu(x)
		=\mathrm{Asc} _\nu^-(\mathcal{S},\alpha ,\mathcal{U}).
	\end{aligned}
	$$\end{proof}
\begin{proposition}\label{prop4.2}
	Let $\varepsilon>0$ and $\mathcal{T},\mathcal{S}\in I(\mathcal{R})$. If $\mathcal{T}\subseteq_{\varepsilon}\mathcal{S}$, then
	$$\mathrm{Asc} _\nu^-(\mathcal{S},\phi ,\mathcal{U})\le \mathrm{Asc} _\nu^-(\mathcal{T},\phi ,\mathcal{U})+3\varepsilon \log N(\mathcal{U}),\mathrm{Asc} _\nu^+(\mathcal{S},\phi ,\mathcal{U})\le \mathrm{Asc} _\nu^+(\mathcal{T},\phi ,\mathcal{U})+3\varepsilon \log N(\mathcal{U}).$$
\end{proposition}

\begin{proof}
	Let $\alpha\in\mathcal{P}_{X\times Y}$. If the first inequation has been proved, then we have $ \mathrm{Asc} _\nu ^-(\mathcal{S} ,\phi ,\alpha)\le\mathrm{Asc} _\nu ^-(\mathcal{T} ,\phi ,\alpha)+3\varepsilon\log N(\alpha)$. Thus
	$$
	\begin{aligned}
		&\mathrm{Asc} _\nu^+(\mathcal{S},\phi ,\mathcal{U})=\inf_{\alpha\in\mathcal{P}_{X\times Y},\alpha \succeq \mathcal{U} }\{\mathrm{Asc} _\nu(\mathcal{S},\phi ,\alpha),N(\alpha)\le N(\mathcal{U})\}\\
		\le&\inf_{\alpha\in\mathcal{P}_{X\times Y},\alpha \succeq \mathcal{U} }\{\mathrm{Asc} _\nu(\mathcal{S},\phi ,\alpha)+3\varepsilon \log N(\mathcal{U}),N(\alpha)\le N(\mathcal{U})\}= \mathrm{Asc} _\nu^+(\mathcal{T},\phi ,\mathcal{U})+3\varepsilon \log N(\mathcal{U}).
	\end{aligned}
	$$
	
	So we only need to show the first inequation. Let $A_0=\{x\in A:\mathcal{T}(x)\subseteq\mathcal{S}(x)\}$. By Lemma \ref{lem4.2}, we have $\mu(A_0)>1-2\varepsilon$ and $A_0$ is $\mathcal{T}$-invariant. We define the maps $f,g:A_0\to\mathbb{R}$ by
	$$f(x)=\frac{1}{|\mathcal{S} (x)\cap A_0|} \sum _{B\subseteq \mathcal{S} (x)\cap A_0}c^{\mathcal{S} (x)\cap A_0}_{B}H_\nu \left ( \bigvee _{y\in B} \phi(x,y)\mathcal{U}_y \right ),$$$$g(x)=\frac{1}{|\mathcal{T} (x)|} \sum _{B\subseteq \mathcal{T} (x)}c^{\mathcal{T} (x)}_{B}H_\nu \left ( \bigvee _{y\in B} \phi(x,y)\mathcal{U}_y \right ). $$
	Since $A_0$ is $\mathcal{T}$-invariant, for each $x\in A_0$, there exist $x_1,\dots,x_k\in X$ such that $\mathcal{S} (x)\cap A_0= \bigsqcup_{i=1}^{k}\mathcal{T}(x_i) $. For each $B\subseteq\mathcal{S} (x)\cap A_0$, we define $B_i=B\cap\mathcal{T}(x_i)$ with $B=\bigsqcup_{i=1}^{k}B_i$. Then
	$$
	\begin{aligned}
		f(x)
		&\le  \frac{1}{|\mathcal{S} (x)\cap A_0|} \sum _{B\subseteq \bigsqcup_{i=1}^{k}\mathcal{T}(x_i) }\sum_{i=1}^kc^{\mathcal{T}(x_i) }_{B_i}H_\nu \left ( \bigvee _{y\in B_i} \phi(x,y)\mathcal{U}_y \right )=\frac{1}{|\mathcal{S} (x)\cap A_0|} \sum_{i=1}^k|\mathcal{T}(x_i)|\cdot g(x_i)\\
		&=\frac{1}{|\mathcal{S} (x)\cap A_0|} \sum_{i=1}^k\sum_{y\in\mathcal{T}(x_i)} g(y)=\mathbb{E}(g|\mathcal{S}\cap (A_o\times A_o))(x).
	\end{aligned}
	$$
	So
	$$
	\begin{aligned}
		\mathrm{Asc} _\nu ^-(\mathcal{S} ,\phi ,\mathcal{U})
		&\le \int\limits_{A_o}\left(f(x)+ \sum _{B\subseteq\mathcal{S} (x)}c^{\mathcal{S} (x)}_B\frac{|B\setminus A_0|}{|\mathcal{S} (x)|}\log N(\mathcal{U})\right)d\mu(x)+2\varepsilon\log N(\mathcal{U})\\
		&\le \int\limits_{A_o}\mathbb{E}(g|\mathcal{S}\cap (A_o\times A_o))(x)d\mu(x)+3\varepsilon\log N(\mathcal{U})\le\mathrm{Asc} _\nu ^-(\mathcal{T} ,\phi ,\mathcal{U})+3\varepsilon\log N(\mathcal{U}).
	\end{aligned}
	$$
	We finish the proof.
\end{proof}
By Lemma \ref{lem4.1} and Proposition \ref{prop4.2}, we have the following remark.
\begin{remark}\label{rmk4.2}
	Let $\mathcal{R} $ be hyperfinite with $\left \{ \mathcal{R} _n \right \} _{n\in \mathbb{N} }$ a filtration of $\mathcal{R}$. Then 
	$$
	\lim_{n \to \infty} \mathrm{Asc} _\nu ^-(\mathcal{R} _n,\phi ,\mathcal{U})=\mathrm{Asc} _\nu ^-(\phi ,\mathcal{U}) \text{ }\mathrm{ , } \lim_{n \to \infty} \mathrm{Asc} _\nu ^+(\mathcal{R} _n,\phi ,\mathcal{U})=\mathrm{Asc} _\nu ^+(\phi ,\mathcal{U}).
	$$
\end{remark}
Let $\mathcal{R}$ be generated by a free $G$-measure preserving system $(X,\mathcal{B}_X,\mu,G)$. Then $\mathcal{R}$ is hyperfinite and  conservative. Let $\mathcal{S}\in I(\mathcal{R})$, so there exists $B\in\mathcal{B}_X$ an $\mathcal{S}$-fundamental i.e. $|B\cap\mathcal{S}(x)|=1 $. Then there is a measurable map $\theta :B\to \mathcal{F}(G)$ with $x\mapsto G_x$ and $G_xx=\mathcal{S}(x)$. Since $\mathcal{F}(G)$ is a countable set, we can show that $X=\sqcup _i\sqcup_{g\in G_i}gB_i$ where $\sqcup $ denotes the disjoint union and $\{G_i\}_i \subseteq\mathcal{F}(G)$ with $G_ix=S(x)$ and $B=\sqcup_i B_i$ for each $x\in B_i$. We write it as  $\mathcal{S}\sim(B_i,G_i)$. So we have

$$
\begin{aligned}
	\mathrm{Asc} _\nu ^-(\mathcal{S},\phi ,\mathcal{U})&=\sum \limits_i\int\limits_{B_i}\sum\limits _{D\subseteq G_i} c^{G_i}_{D}H_\nu\left( \bigvee\limits _{g\in D}\phi(x,gx)\mathcal{U}_{gx}\right)d\mu(x).
\end{aligned}
$$

\begin{definition}
	Let $(Y,\mathcal{B}_Y,\nu,G)$ be a $G$-measure preserving system, $\mathcal{U}\in\mathcal{C}_Y$, $\prod _g\in Aut(Y,\nu)$ the action of $g\in G$ on $(Y,\mathcal{B}_Y,\nu,G)$ and $\phi _G:\mathcal{R}\to   Aut(Y,\nu)$ a cocycle with $\phi _G(gx,x)=\prod _g$ for any $x\in X, g\in G$. We define \textit{the $\nu^+$$(\nu^-)$-virtual average sample complexity of $\mathcal{U}$} by 
	$$
	\widehat{\mathrm{Asc}}_\nu ^-(G ,\mathcal{U}):=\mathrm{Asc} _\nu ^-(\phi_G ,X\times\mathcal{U})\text{ }
	\left(\widehat{\mathrm{Asc}}_\nu ^+(G ,\mathcal{U}):=\mathrm{Asc} _\nu ^+(\phi_G ,X\times\mathcal{U})\right).
	$$
\end{definition}

\begin{remark}
	It is easy to show that $\widehat{\mathrm{Asc}}_\nu ^-(G,\alpha)=\widehat{\mathrm{Asc}}_\nu ^+(G,\alpha)$ for each $\alpha\in\mathcal{P}_Y$ and $\widehat{\mathrm{Asc}}_\nu ^+(G,\mathcal{U})=\inf_{\alpha\in\mathcal{P}_{Y},\alpha \succeq \mathcal{U} }\widehat{\mathrm{Asc}}_\nu ^-(G ,\alpha)$ for each $\mathcal{U}\in\mathcal{C}_Y$.
	The proof of the well-defined $\nu^+$ and $(\nu^-)$-virtual average sample complexity of $\mathcal{U}$ can be referred to \cite[Proposition 4.8]{huang2011local}.
\end{remark}
\begin{lemma}\label{lem n}
	Let $(Y,\mathcal{B}_Y,\nu,G)$ be a $G$-measure preserving system, $\mathcal{U}\in\mathcal{C}_Y$ and $\varepsilon>0$. There exist $E\in \mathcal{F}(G)$ and $0<\varepsilon'<\varepsilon$ such that if $F$ is $(E,\varepsilon')$-invariant, then $$\left|\frac{1}{|F|} \sum _{S\subseteq F}c^F_SH_\nu(\mathcal{U} _S)-\mathrm{Asc} _\nu (G,\mathcal{U})\right|<\varepsilon.$$
\end{lemma}
\begin{proof}
	The detailed proof can be refer to \cite{huang2011local}.
\end{proof}
\begin{theorem}\label{thm4.4}
	Let $(Y,\mathcal{B}_Y,\nu,G)$ be a $G$-measure preserving system and $\mathcal{U}\in\mathcal{C}_Y$. Then
	$$\mathrm{Asc}_\nu ^-(G,\mathcal{U})=\widehat{\mathrm{Asc}}_\nu ^-(G,\mathcal{U})\text{ }\mathrm{and } \text{ }\mathrm{Asc}_\nu ^+(G,\mathcal{U})=\widehat{\mathrm{Asc}}_\nu ^-(G,\mathcal{U}).$$
	
\end{theorem}
\begin{proof}
	Let $\{\mathcal{R}_n\}_{n\in\mathbb{N}}$ be a filtration of $\mathcal{R}$ with $\mathcal{R}_n\sim (B_i^{(n)},G_i^{(n)})$ for each $n\in\mathbb{N}$. By the definition of $\mathrm{Asc}_\nu ^-(\phi_G,X\times \mathcal{U})$ and Lemma \ref{lem4.3}, for each $n$ large enough there exists a measurable $\mathcal{R}_n$-invariant subset $A_n\subseteq X$ such that $\mu(A_n)>1-\varepsilon'$ and 
	$$
	\left | \left \{ y\in \mathcal{R}_n(x):Ey\subseteq\mathcal{R}_n(x)  \right \}  \right | >(1-\varepsilon ')|\mathcal{R}_n(x)|\text{ for each } x\in A_n\text{, }E\in \mathcal{F}(G).
	$$ 
	Now we prove that
	$
	|\mathrm{Asc} _\nu ^-(\mathcal{R}_n,\phi_G, X\times \mathcal{U})-\mu(A_n)\mathrm{Asc} _\nu ^-(G,\mathcal{U})|\to0\text{, as }n\to\infty\text{, }\varepsilon\to0.
	$
	
    By Lemma \ref{lem n}, for each $\varepsilon>0$ there exist $E\in \mathcal{F}(G)$ and $0<\varepsilon'<\varepsilon$ such that if $F$ is $(E,\varepsilon')$-invariant, then $$|\frac{1}{|F|} \sum _{S\subseteq F}c^F_SH(\mathcal{U} _S)-\mathrm{Asc} _\nu (G,\mathcal{U})|<\varepsilon.$$ Since $A_n$ is $\mathcal{R}_n$-invariant, i.e. $A_n=\sqcup _{i\in J}G_i^{(n)}C_i^{(n)}$ for some subset $J\subseteq\mathbb{N}$ and measurable subsets $C_i^{(n)}\subseteq B_i^{(n)}$ satisfying $\mu(C_i^{(n)})>0$, $i\in J$. If $i\in J$, $x\in C_i^{(n)}$ and $g\in G_i^{(n)}$, we have
	$$
	(1-\varepsilon ')|\mathcal{R}_n(gx)|<\left | \left \{ y\in \mathcal{R}_n(gx):Ey\subseteq\mathcal{R}_n(gx)  \right \}  \right | =\left | \left \{ y\in \mathcal{R}_n(x):Ey\subseteq\mathcal{R}_n(x)  \right \}  \right | .
	$$ 
	So $(1-\varepsilon' )|G_i^{(n)}|<|\{g\in G_i^{(n)}:Eg\subseteq G_i^{(n)}\}|$, that is $G_i^{(n)}$ is $(E,\varepsilon')$-invariant. We define
	$$
	f(x)=\frac{1}{|\mathcal{R}_n(x)|} \sum _{B\subseteq \mathcal{R}_n(x)}c_B^{\mathcal{R}_n(x)}H_\nu \left(\bigvee _{y\in B}\phi_G(x,y)\mathcal{U}\right)\text{ for each } x\in X.
	$$
	It is easy to show that $f_n(x)\le \log N(\mathcal{U})$. So
	$$
	\int\limits_{A_n}f(x)d\mu (x)=\sum_{j\in J}\int\limits_{C_j^{(n)}}\sum_{D\subseteq {G_i}}c^{G_i}_DH_\nu \left(\bigvee _{g\in D} {\textstyle \prod_{g}^{-1}} \mathcal{U}\right)d\mu(x).
	$$
	So we have 
	$$
	\begin{aligned}
		&\left|\mathrm{Asc} _\nu ^-(\mathcal{R}_n,\phi_G, X\times \mathcal{U})-\mu(A_n)\mathrm{Asc} _\nu ^-(G,\mathcal{U})\right|
		\le\left|\int\limits_{A_n}\left(f(x)-\mathrm{Asc} _\nu ^-(G,\mathcal{U})\right)d\mu (x)+\int\limits_{X\setminus{A_n}}f(x)d\mu (x)\right|\\
		&\le\left( \sum\limits_{j\in J}|C_j^{(n)}|\mu(C_j^{(n)})\right)\varepsilon+(1-\mu(A_n))\log N(\mathcal{U})\to 0,\text{ as } n\to \infty \text{ and } \varepsilon\to0.
	\end{aligned}
	$$
	We obtain that $\mathrm{Asc}_\nu ^-(G,\mathcal{U})=\widehat{\mathrm{Asc}}_\nu ^-(G,\mathcal{U})$.
	Additionally, 
	$$
	\widehat{\mathrm{Asc}}_\nu ^+(G,\mathcal{U})=\inf_{\alpha\in\mathcal{P}(X)\alpha\succeq\mathcal{U}}\widehat{\mathrm{Asc}}_\nu (G,\alpha)=\inf_{\alpha\in\mathcal{P}(X)\alpha\succeq\mathcal{U}}\mathrm{Asc}_\nu (G,\alpha)=\mathrm{Asc}_\nu ^+(G,\mathcal{U}).
	$$
	The proof is finished.
\end{proof}

\begin{theorem}\label{thm4.5}
	Let $\gamma$ be an invertible measure-preserving transformation on $(X,\nu)$ generating $\mathcal{R}$, $\phi:\mathcal{R}\to Aut(Y,\nu)$ a cocycle, $\mathcal{U}\in\mathcal{C}_{X\times Y}$ and $\gamma_\phi$ be the $\phi$-skew product extension of $\gamma$. Then 
	$$
	\mathrm{Asc}_\nu ^-(\phi,\mathcal{U})=\mathrm{Asc}_{\mu\times\nu} ^-(\gamma_\phi,\mathcal{U}|\mathcal{B}_X\otimes \{\emptyset ,Y\})\text{, }\mathrm{Asc}_\nu ^+(\phi,\mathcal{U})=\mathrm{Asc}_{\mu\times\nu}^+(\gamma_\phi,\mathcal{U}|\mathcal{B}_X\otimes \{\emptyset ,Y\}).
	$$
\end{theorem}
\begin{proof}
    The proof is similar to \cite{huang2011local}.
\end{proof}
Combining above theorems in this section, we will obtain the following theorem.
\begin{theorem}\label{thm4.6}
	Let $(Y,\mathcal{B}_Y,\nu,G)$ be a $G$-measure preserving system with $(Y,\mathcal{B}_Y,\nu)$ a Lebesgue space and $\mathcal{U}\in\mathcal{C}_Y$. Then $\mathrm{Asc}_\nu ^-(G,\mathcal{U})=\mathrm{Asc}_\nu ^+(G,\mathcal{U})$.
\end{theorem}
\begin{proof}
	Let $(X,\mathcal{B}_X,\mu,G)$ be a free $G$-measure preserving system with $\mathcal{R}\subseteq X\times X$ the $G$-orbit equivalence relation and $\gamma$ an invertible measure-preserving transformation on $(X,\mathcal{B}_X,\mu)$ generating $\mathcal{R}$. Let $\phi_G: \mathcal{R}\to Aut(Y,\nu)$ be a cocycle satisfying $\phi_G(gx,x)=\prod _g$, where $\prod _g\in Aut(Y,\nu)$ the action of $g\in G$ on $(Y,\mathcal{B}_Y,\nu,G)$. By Theorem \ref{thm4.4}, we have
	
	$$\mathrm{Asc}_\nu ^-(G,\mathcal{U})=\widehat{\mathrm{Asc}}_\nu ^-(G,\mathcal{U})=\mathrm{Asc} _\nu ^-(\phi_G ,X\times\mathcal{U})\text{, }\mathrm{Asc}_\nu ^+(G,\mathcal{U})=\widehat{\mathrm{Asc}}_\nu ^-(G,\mathcal{U})=\mathrm{Asc} _\nu ^+(\phi_G ,X\times\mathcal{U}).$$
	By Theorem \ref{thm4.5}, we have 
	$$
	\mathrm{Asc}_\nu ^-(\phi,\mathcal{U})=\mathrm{Asc}_{\mu\times\nu} ^-(\gamma_\phi,\mathcal{U}|\mathcal{B}_X\otimes \{\emptyset ,Y\})\text{, }\mathrm{Asc}_\nu ^+(\phi,\mathcal{U})=\mathrm{Asc}_{\mu\times\nu}^+(\gamma_\phi,\mathcal{U}|\mathcal{B}_X\otimes \{\emptyset ,Y\}).
	$$
	As $\mathcal{B}_X\otimes \{\emptyset ,Y\}$ is $T$-invariant, combining above two equations and Theorem \ref{thm4.5}, we have
	$$
	\mathrm{Asc}_{\mu\times\nu} ^-(\gamma_\phi,\mathcal{U}|\mathcal{B}_X\otimes \{\emptyset ,Y\})=\mathrm{Asc}_{\mu\times\nu}^+(\gamma_\phi,\mathcal{U}|\mathcal{B}_X\otimes \{\emptyset ,Y\}).
	$$
	So $\mathrm{Asc}_\nu ^-(G,\mathcal{U})=\mathrm{Asc}_\nu ^+(G,\mathcal{U})$. We finish the proof.
\end{proof}

\section{A local variational principle of average sample complexity}\label{sec5}

In this section, we aim to present a local variational principle of average sample complexity of amenable group actions. First, we review the theorem on variational principle of average sample complexity.

\begin{theorem}
	(Variational principle of average sample complexity) Let $\mathcal{U}\in\mathcal{C}^o_X$. Then 
	$$
	\mathrm{Asc}_{top} (G,X)=\sup_{\mu\in\mathcal{M}(X,G)}\mathrm{Asc}_\mu (G,X)=\sup_{\mu\in\mathcal{M}^e(X,G)}\mathrm{Asc}_\mu (G,X).
	$$
\end{theorem}
\begin{proof}
	Li et al.(Theorem 3.6, \cite{li2021dynamical}) showed that for each $\mu\in\mathcal{M}(X,G)$, 
	\begin{equation}\label{eq6.4}
		\mathrm{Asc}_{top} (G,X)=h_{top} (G,X)\text{ and }\mathrm{Asc}_\mu (G,X)=h_\mu (G,X).
	\end{equation}
	Huang et al. (Theorem 5.1, \cite{huang2011local}) give the variational principle of topological entropy. That is,
	\begin{equation}\label{eq6.2}
		h _\mathrm{top}(G,X)=\sup_{\mu\in\mathcal{M}(X,G)}h_\mu (G,X)=\sup_{\mu\in\mathcal{M}^e(X,G)}h_\mu (G,X).
	\end{equation}
	Combining with Equations \eqref{eq6.4} and \eqref{eq6.2}, we prove the theorem.
	
\end{proof}

By Theorem \ref{thm4.6}, we define $\mathrm{Asc}_\nu (G,\mathcal{U}):=\mathrm{Asc}_\nu ^-(G,\mathcal{U})=\mathrm{Asc}_\nu ^+(G,\mathcal{U})$.
\begin{theorem}
	(Local variational principle of average sample complexity) Let $\mathcal{U}\in\mathcal{C}^o_X$ and $\{c^{F_n}_S=2^{-|F_n|}: S\subseteq F_n\}$ a system of coefficients. Then 
	$$
	\mathrm{Asc}_{top} (G,\mathcal{U})=\max_{\mu\in\mathcal{M}(X,G)}\mathrm{Asc}_\mu (G,\mathcal{U})=\max_{\mu\in\mathcal{M}^e(X,G)}\mathrm{Asc}_\mu (G,\mathcal{U}).
	$$
\end{theorem}
\begin{proof}
	By Theorem \ref{thm4.6}, we have $ \mathrm{Asc}_{top} (G,\mathcal{U})\ge\mathrm{Asc}_\mu (G,\mathcal{U})$ for each $\mu\in\mathcal{M}(X,G)$. By Theorem \ref{thm3.2}, there exists $\mu\in\mathcal{M}(X,G)$ such that $ \mathrm{Asc}_{top} (G,\mathcal{U})\le\mathrm{Asc}_\mu (G,\mathcal{U})$.
	So $$\mathrm{Asc}_{top} (G,\mathcal{U})=\max_{\mu\in\mathcal{M}(X,G)}\mathrm{Asc}_\mu (G,\mathcal{U}) .$$
	Now we  prove  that there exists $\nu\in\mathcal{M}^e(X,G)$ such that $ \mathrm{Asc}_{top} (G,\mathcal{U})\le\mathrm{Asc}_\nu (G,\mathcal{U})$.
	
	Let $\mu=\int_{\mathcal{M}^e(X,G)}\mu_\omega dm(\omega )$ be the ergodic decomposition of $\mu$. By Theorem \ref{thm3.4}, we have 
	$$\int\limits_{\mathcal{M}^e(X,G)}\mathrm{Asc}_{\mu_\omega }(G,\mathcal{U})dm(\omega )=\mathrm{Asc}_\mu (G,\mathcal{U})\ge \mathrm{Asc}_{top} (G,\mathcal{U}).$$
	So $\mathrm{Asc}_{top} (G,\mathcal{U})\le\mathrm{Asc}_{\mu_\omega }(G,\mathcal{U})$ for some ${\mu_\omega }\in\mathcal{M}^e(X,G)$. Then the proof is completed.

\end{proof}

%The acknowledgments section should not be numbered.
\section*{Acknowledgments}
We would like to thank Prof. Zhao Yun for his useful suggestions and instructions. The corresponding author's research was supported by the National Natural Science Foundation of China (12201120) and the visiting fellowships supported by Fujian Alliance of Mathematics.
\section*{\textbf{Conflict of interests}}%
On behalf of all authors, the corresponding author states that there is no conflict of interests.

%%%%%%%%%%%%%%%%%%%%%%%%%%%%%%%%%%%%%%%%%%%%%%%%%%%%%%
%          7. REFERENCES SECTION
%%%%%%%%%%%%%%%%%%%%%%%%%%%%%%%%%%%%%%%%%%%%%%%%%%%%%%

%       READ THIS SECTION CAREFULLY

% Each of the references below MUST be cited in your article above. Do not include references that are not cited in your article.

% Follow the examples below carefully. We strongly suggest that you copy and paste your reference information directly into our examples.

% List all references in alphabetical order according to the first author's last name.

% Verify each URL works correctly and can be accessed properly. Your URL links should be to reputable websites. The command line for a website link begins with: \url{ }

% Do not add MR or DOI numbers to your references. AIMS production staff will add this information.

% Using BibTex is not recommended but can be handled.


\begin{thebibliography}{99}

% Work in Progress Example:

\bibitem{adler1965topological}
R.L. Adler{,} A.G. Konheim{,}~M.H. McAndrew.
\newblock {\em Topological entropy}.
\newblock {Trans Amer Math Soc},
114(2):309--319, 1965.

\bibitem{blanchard1997variation}
F.~Blanchard{,}~E. Glasner{,}~B.~Host.
\newblock {\em A variation on the variational principle and applications to entropy pairs.}
\newblock {Ergodic Theory Dynam System}, 17(1):29--43, 1997.


\bibitem{bowen1975equilibrium}
R.~Bowen.
\newblock {\em Equilibrium states and the ergodic theory of anosov diffeomorphisms.}
\newblock {Lecture notes in math}, 470:1--17, 1975.

\bibitem{bowen1979hausdorff}
R.~Bowen.
\newblock {\em Hausdorff dimension of quasi-circles.}
\newblock {Publ Math IHES}, 50:11--25, 1979.

\bibitem{buzzi2012approximate}
J.~Buzzi{,}~L. Zambotti.
\newblock {\em Approximate maximizers of intricacy functionals.}
\newblock {Probab Theory and Related Fields}, 153:421--440, 2012.

\bibitem{buzzi2012mean}
J.~Buzzi{,}~L. Zambotti.
\newblock {\em Mean mutual information and symmetry breaking for finite random fields.}
\newblock {Ann Inst H Poincar´e Probab Statist},48:343--367, 2012.
\bibitem{cao2008thermodynamic}

Y.~Cao{,} D. Feng{,}~W. Huang.
\newblock {\em The thermodynamic formalism for sub-additive potentials.}
\newblock {Disc Cont Dyn Syst}, 20(3):639, 2008.

\bibitem{danilenko2001entropy}
A.I.~Danilenko.
\newblock {\em Entropy theory from the orbital point of view.}
\newblock {Monatsh Math}, 134(2):121--141, 2001.

\bibitem{eickhoff2018imaging}
S.B. Eickhoff{,} B.T.T. Yeo{,}~S. Genon.
\newblock {\em Imaging-based parcellations of the human brain.}
\newblock {Nature Reviews Neuroscience}, 19(11):672--686, 2018.


\bibitem{feldman1977ergodic}
J.~Feldman{,} C. Moore.
\newblock {\em Ergodic equivalence relations, cohomology, and von Neumann algebras. I}.
\newblock {Trans Amer Math Soc}, 234(2):289--324, 1977.


\bibitem{fornito2016fundamentals}
A.~Fornito{,} A. Zalesky{,}~E. Bullmore.
\newblock {\em Fundamentals of brain network analysis}.
\newblock Academic Press, 2016.

\bibitem{glasner2003ergodic}
E.~Glasner.
\newblock {\em Ergodic theory via joinings.}
\newblock {Amer Math Soc}, 2003.

\bibitem{golodets1994classification}
V.Y. Golodets{,}~S.D. Sinelshchikov.
\newblock {\em Classification and structure of cocycles of amenable ergodic equivalence relations.}
\newblock {J Funct Anal}, 121(2):455--485, 1994.

\bibitem{keller1998equilibrium}
G.~Keller.
\newblock {\em Equilibrium states in ergodic theory}, volume~42.
\newblock {Cambridge University Press}, 1998.



\bibitem{huang2004entropy}
W.~Huang{,} A.~P. Maass{,}~P.P.~Romagnoli{,} X.~Ye.
\newblock {\em Entropy pairs and a local Abramov formula for a measure theoretical entropy of open covers.}
\newblock {Ergodic Theory Dynam System}, 24(4):1127--1153, 2004.


\bibitem{HUANG_YE_ZHANG_2006}
W.~Huang{,} X. Ye{,}~G. Zhang.
\newblock {\em A local variational principle for conditional entropy.}
\newblock {Ergodic Theory Dynam System}, 26(1):219-245, 2006.

\bibitem{huang2011local}
W.~Huang{,}~X. Ye{,}~G.~Zhang.
\newblock {\em Local entropy theory for a countable discrete amenable group action.}
\newblock {J Funct Anal}, 261(4):1028--1082, 2011.






\bibitem{kerr2016ergodic}
D.~Kerr{,}~H. Li.
\newblock {\em Ergodic Theory-Independence and Dichotomies}.
\newblock {Springer}, 2016.







\bibitem{li2021dynamical}
J.~Li{,}~S. Tu.
\newblock {\em Dynamical intricacy and average sample complexity of amenable group actions.}
\newblock {Sci China Math}, 6:1--20, 2021.

\bibitem{liang2012topological}
B.~Liang{,}~K. Yan.
\newblock {\em Topological pressure for sub-additive potentials of amenable group actions.}
\newblock {J Funct Anal}, 262(2):584--601, 2012.

\bibitem{ollagnier2007ergodic}
J.~Moulin{,}~J. Ollagnier.
\newblock {\em Ergodic theory and statistical mechanics}, volume 1115.
\newblock Springer, 2007.

\bibitem{ornstein1987entropy}
D.S. Ornstein{,}~B. Weiss.
\newblock {\em Entropy and isomorphism theorems for actions of amenable groups.}
\newblock {J Anal Math}, 48(1):1--141, 1987.

\bibitem{ornstein1980ergodic}
D.S. Ornstein{,}~B. Weiss.
\newblock {\em Ergodic theory of amenable group actions I: The Rohlin lemma.}
\newblock Bull Amer Math Soc NS, 2: 161--164, 1980.

\bibitem{petersen2018dynamical}
K.~Petersen{,}~B. Wilson.
\newblock {\em Dynamical intricacy and average sample complexity.}
\newblock {Dyn Syst}, 33(3):369--418, 2018.

\bibitem{romagnoli2018local}
P.P.~Romagnoli.
\newblock {\em Local conditional entropy in measure for covers with respect to a fixed partition.}
\newblock {Nonlinearity}, 31(5):2201--2220, 2018.


\bibitem{ruelle1982repellers}
D.~Ruelle.
\newblock {\em Repellers for real analytic maps.}
\newblock {Ergodic Theory Dynam System}, 2(1):99--107, 1982.

\bibitem{ruelie1973statistical}
D.~Ruelle.
\newblock {\em Statistical mechanics on a compact set with $\mathbb{Z}^\nu$ action satisfying expansiveness and specification.}
\newblock {Trans Amer Math Soc}, 185:237--251, 1973.

\bibitem{ruelle2004thermodynamic}
D.~Ruelle.
\newblock {\em Thermodynamic formalism: the mathematical structure of equilibrium statistical mechanics}.
\newblock Cambridge University Press, 2004.

\bibitem{tononi1994measure}
G.~Tononi{,} O. Sporns{,}~G.M. Edelman.
\newblock {\em A measure for brain complexity: relating functional segregation and integration in the nervous system.}
\newblock {Proc Natl Acad Sci}, 91(11):5033--5037, 1994.

\bibitem{walters1982introduction}
P.~Walters.
\newblock {\em An introduction to ergodic theory.}
\newblock {Springer-Verlag}, 1982.

\bibitem{weiss1985strictly}
B.~Weiss.
\newblock {\em Strictly ergodic models for dynamical systems.}
\newblock {Bull Amer Math Soc}, 13:143--146, 1985.

\bibitem{weiss2001monotileable}
B.~Weiss.
\newblock {\em Monotileable amenable groups.}
\newblock {Trans Amer Math Soc}, 202(2):257--262, 2001.

\bibitem{xiao2024pressure}
Z.~Xiao{,}~J. Huang.
\newblock {\em The pressure of intricacy and average sample complexity for amenable group actions.}
\newblock {Monatsh Math}, 205:391--414, 2024.

\bibitem{yang2022dynamical}
K. Yang{,}~E. Chen{,}~X. Zhou.
\newblock {\em Dynamical intricacy and average sample complexity for random bundle transformations.}
\newblock {Journal of Mathematical Physics}, 63(1):012701, 2022.


\bibitem{ye2007entropy}
X. Ye{,}~G. Zhang.
\newblock {\em Entropy points and applications.}
\newblock {Trans Amer Math Soc}, 359(12):6167--6186, 2007.





\end{thebibliography}
\end{document}